\documentclass[english]{amsart}
\usepackage{babel}
\usepackage[]{geometry}
\usepackage{caption}

\usepackage{amsmath,amssymb,amsfonts,latexsym,cancel}
\usepackage{rawfonts}
\usepackage{pictexwd}
\usepackage{tikz}

\usepackage{color}
\usepackage{epsfig}

\newcommand{\R}{\mathbb{R}}
\newcommand{\N}{\mathbb{N}}

\def\1{\raisebox{2pt}{\rm{$\chi$}}}

\newcommand{\ud}{\, d}

\newcommand{\abs}[1]{\left| #1 \right|}

\newcommand{\I}{\textrm{I}}
\newcommand{\II}{\textrm{II}}
\newcommand{\dist}{\operatorname{dist}}

\newcommand{\eps}{\varepsilon}

\theoremstyle{plain}
\newtheorem{definition}{Definition}[section]
\newtheorem{proposition}[definition]{Proposition}
\newtheorem{theorem}[definition]{Theorem}

\newtheorem{lemma}[definition]{Lemma}

\newtheorem{remark}[definition]{Remark}

\theoremstyle{definition}
\newtheorem{ejem}{Example}

\theoremstyle{remark}

\numberwithin{equation}{section}

\begin{document}

\title[The evolution problem associated with
eigenvalues of the Hessian]
{\bf The evolution problem associated with
eigenvalues of the Hessian}

\author[P. Blanc, C. Esteve and J. D. Rossi]{Pablo Blanc, Carlos Esteve and Julio D. Rossi}

\address{P. Blanc and J. D. Rossi
\hfill\break\indent
Depto. Matem\'{a}tica, FCEyN, Buenos Aires University,
\hfill\break\indent
Ciudad
Universitaria, Pab~1~(1428),
\hfill\break\indent
Buenos Aires, Argentina.
\hfill\break\indent
 {\tt pblanc@dm.uba.ar, jrossi@dm.uba.ar}}

\address{C. Esteve
\hfill\break\indent Universit\'e Paris 13, Sorbonne Paris Cit\'e, 
\hfill\break\indent
Laboratoire Analyse, Geometrie et Applications, 
\hfill\break\indent 93430, Villetaneuse, France.
\hfill\break\indent
{\tt esteve@math.univ-paris13.fr}}

\date{\today}


\begin{abstract}
In this paper we study the evolution problem
\[
\left\lbrace\begin{array}{ll}
u_t (x,t)- \lambda_j(D^2 u(x,t)) = 0, & \text{in } \Omega\times (0,+\infty), \\
u(x,t) = g(x,t), & \text{on } \partial \Omega \times (0,+\infty), \\
u(x,0) = u_0(x), & \text{in } \Omega,
\end{array}\right.
\]
where $\Omega$ is a bounded domain in $\mathbb{R}^N$ (that verifies a suitable geometric condition on its boundary) and $\lambda_j(D^2 u)$ stands
for the $j-$st eigenvalue of the Hessian matrix $D^2u$. We assume that $u_0 $ and $g$
are continuous functions with the compatibility condition $u_0(x) = g(x,0)$, $x\in \partial \Omega$.

We show that the (unique) solution to this problem exists in the viscosity sense and can be approximated by 
the value function of a two-player zero-sum game as 
the parameter measuring the size of the step that we move in each round of the game goes to zero.  

In addition, when the boundary datum is independent of time, $g(x,t) =g(x)$, 
we show that viscosity solutions to this evolution problem stabilize and converge exponentially fast to the unique stationary solution 
as $t\to \infty$. For $j=1$ the limit profile is just the convex envelope inside $\Omega$ of the boundary datum $g$,
while for $j=N$ it is the concave envelope. We obtain this result with two different techniques: with PDE tools and 
and with game theoretical arguments. Moreover, in some special cases (for affine boundary data) 
we can show that solutions coincide with the stationary solution in finite time (that depends only on $\Omega$ and not on 
the initial condition $u_0$). 
\end{abstract}

\maketitle

\section{Introduction}\label{sec:intro}
\setcounter{equation}{0}

Consider the problem
\begin{equation}\label{convex envelope evolution}
\left\lbrace\begin{array}{ll}
u_t (x,t) - \lambda_j(D^2 u(x,t)) = 0, & \text{in } \Omega\times (0,+\infty), \\
u(x,t) = g(x,t), & \text{on } \partial \Omega \times (0,+\infty), \\
u(x,0) = u_0(x), & \text{in } \Omega.
\end{array}\right.
\end{equation}
here $\Omega$ is a bounded domain in $\mathbb{R}^N$, with $N\geq 1$ and 
$\lambda_j(D^2 u)$ stands for the $j-$th eigenvalue of $D^2u = (\partial^2_{x_i,x_j} u)_{ij}$, which is the hessian matrix of $u$.
We will assume from now on that $u_0$ and $g$
are continuous functions with the compatibility condition $u_0(x) = g(x,0)$, $x\in \partial \Omega$.

Problem \eqref{convex envelope evolution} is the evolution version of the elliptic problem
\begin{equation}\label{convex envelope equation}
\left\lbrace\begin{array}{ll}
\lambda_j(D^2 z(x)) = 0, & \text{in } \Omega, \\
z(x) = g(x), & \text{on } \partial \Omega,
\end{array}\right.
\end{equation}
which was extensively studied in \cite{Alex,Birin,Birin2,BirindelliIshii,BlancRossi,caffa,HL1,HL2,OS,Ober}. In particular, for $j=1$ and $j=N$, problem \eqref{convex envelope equation} is the equation for the convex and concave envelope of $g$ in $\Omega$, respectively, i.e., 
the solution $z$ is the biggest convex (smallest concave) function $u$, satisfying $u\leq g$ ($u\geq g$) on $\partial \Omega$, see
\cite{OS,Ober}.

In \cite{BlancRossi}, existence and uniqueness of a continuous solution for \eqref{convex envelope equation} is proved under a hypothesis on the geometry of the domain. Moreover, from the results 
in \cite{HL1} a comparison principle holds for viscosity sub and supersolutions of \eqref{convex envelope equation}. 
Using this comparison principle, together with the connection with concave/convex envelopes of the boundary datum
$g$ for solutions to  \eqref{convex envelope equation}, the geometric condition introduced in \cite{BlancRossi}
turns out to be necessary and sufficient for the well posedness of this problem in the viscosity sense.
In our parabolic setting, using classical ideas from \cite{CIL} one can show that there is also a
comparison principle. Hence, uniqueness of a viscosity solution follows. 
Existence of solutions to \eqref{convex envelope equation} was shown in \cite{HL1}
using Perron's method.
A different existence proof was given in \cite{BlancRossi} where the authors introduce a 
 two-player zero-sum game whose value function approximates the solution of the PDE as the size of the game step goes to zero.

For our parabolic problem, in order to show existence of a continuous viscosity solution it seems natural to try to use
Perron's method relying on the comparison principle. However, we prefer to 
take a different approach. We provide an existence proof using an approximation based on game theory (this approach will be very useful
since it allows us to gain some intuition that will be used when dealing with the asymptotic behaviour of the solutions). For references concerning
games (Tug-of-War games) and fully nonlinear PDEs we refer to \cite{BPR,code,heino,JPR,luiro,MPR,MPRa,MPRb,QS,Par,Peres,PSSW} and to \cite{Delpe,MPRparab} for parabolic versions. 
Here we propose a parabolic version of the game introduced in \cite{BlancRossi} in order to show existence of a viscosity solution to \eqref{convex envelope evolution}. Like for the elliptic problem, it is a two-player zero-sum game. The initial position of the game is determined by a token placed at some point $x_0\in \Omega$ and at some time $t_0>0$. Player I, who wants to minimize the final payoff, chooses a subspace $S$ of dimension $j$ in $\mathbb{R}^N$ and then, Player II, who wants to maximize the final payoff, chooses a unitary vector $v\in S$.  
Then, for a fixed $\eps>0$, the position of the token
is moved to $(x_0+ \eps v,t_0-\eps^2/2)$ or to $(x_0 - \eps v, t_0-\eps^2/2)$ with equal probabilities. 
After the first round, the game continues from the new position $(x_1,t_1)$ according to the same rules. Notice that we take
$t_1 = t_0-\eps^2/2$, but $x_1 = x_0 \pm \eps v$ depends on a coin toss.
The game ends when the token leaves $\Omega\times (0,T]$.

A function $h$ is defined outside the domain.
For our purposes we choose $h$ to be such that $h(x,t)=g(x,t)$ for $x\in\partial\Omega$ and $t>0$, and $h(x,0)=u_0(x)$ for $x\in\Omega$.
That is, $h$ is a continuous extension of the boundary data.
We denote by $(x_\tau,t_\tau )$ the point where the token leaves the domain, that is, either $x_\tau \not\in \Omega$ with $t_\tau >0$, or
$t_\tau\leq 0$.
At this point the game ends and the final payoff is given by $h(x_\tau,t_\tau)$.
That is, Player I pays Player 2 the amount given by $h(x_\tau,t_\tau)$.

For Player I, we denote by $S_I$ a strategy, which is a collection of measurable mappings $S_I = \{ S_k\}_{k=0}^\infty$, where each mapping has the form 
$$
\begin{array}{cccc}
S_k: & \Omega^{k+1} \times \big(k\eps^2/2, +\infty \big) & \longrightarrow & Gr(j, \mathbb{R}^N) \\
     & (x_0,\ldots,x_{k},t_0)  & \longmapsto & S,
\end{array}
$$   
where $S$ is a subspace of dimension $j$.
For Player II, a strategy $S_{II}$ is a collection of measurable mappings $S_{II} = \{ S_k\}_{k=0}^\infty$, where each mapping has the form
$$
\begin{array}{cccc}
S_k: & \Omega^{k+1} \times Gr(j, \mathbb{R}^N) \times \big(k\eps^2/2, +\infty \big) & \longrightarrow & S \\
     & (x_0,\ldots,x_{k},S,t_0)  & \longmapsto & v,
\end{array}
$$
where $v$ is a unitary vector in $S$.

Once both players chose their strategies, we can compute the expected value for the final payoff, which we denote by
$$
\mathbb{E}_{S_I,S_{II}}^{x_0,t_0}[h(x_\tau,t_\tau)].
$$
The value of the game for each player is the best expected value of the final payoff using one of their respective strategies. Since Player I wants to minimize the final payoff and Player II wants to maximize it, we can write the value of the game for each player as follows.
$$
u_I^\eps (x_0,t_0) = \displaystyle\inf_{S_I} \displaystyle\sup_{S_{II}} \mathbb{E}_{S_I,S_{II}}^{x_0,t_0}[h(x_\tau,t_\tau)], \qquad
u_{II}^\eps (x_0,t_0) = \displaystyle\sup_{S_{II}} \displaystyle\inf_{S_I} \mathbb{E}_{S_I,S_{II}}^{x_0,t_0}[h(x_\tau,t_\tau)].
$$
Observe that the expectations above are well defined since the number steps of the game is at most $\lceil 2 t_0/\eps^2\rceil$,
and therefore, the game ends in a finite number of steps with probability 1.
For this game it holds that $u_{I}^\eps (x_0,t_0) = u_{II}^\eps (x_0,t_0)$. Then, we define the value of the game as
$$
u^\eps (x_0,t_0) = u_{I}^\eps (x_0,t_0) = u_{II}^\eps (x_0,t_0).
$$

In section \ref{sect-games} we prove that the game has a value $u^\eps(x,t)$ that verifies an equation
(called the Dynamic Programming Principle in the literature) and that $u^\eps(x,t)$ converges uniformly in $\Omega\times[0,T]$ for every $T>0$ to a function $u(x,t)$, which is continuous and is the unique viscosity solution of the problem \eqref{convex envelope evolution}. This is the content of our first result, see Theorem \ref{teo.convergencia.juego}
below. For the convergence of $u^\eps(x,t)$ we need to assume a condition on the domain that we impose from now on and reads as follows: 
For every $y\in\partial\Omega$, we assume that there exists
$r>0$ such that for every $\delta>0$ there exists $T \subset\R^N$ a subspace of dimension $j$, $w\in\R^N$ of norm 1,
$\lambda>0$ and $\theta>0$ such that
\begin{equation} \label{cond.geoj.intro}
\tag{$F_j$}
\{x\in\Omega\cap B_r(y)\cap T_\lambda: \langle w,x-y\rangle<\theta\}\subset B_\delta(y)
\end{equation}
where 
\[
T_\lambda=\{x\in\R^N: d(x-y,T)<\lambda\}.
\]
As in \cite{BlancRossi}, for our game with a given $j$ we will assume that $\Omega$ satisfies both ($F_j$) and ($F_{N-j+1}$).
Notice that a uniformly convex domain verifies this condition for every $j\in \{1,..,N\}$, but more general domains also satisfy this hypothesis, see \cite{BlancRossi}.

\begin{theorem} \label{teo.convergencia.juego} There is a value function for the game described
before, $u^\eps$. This value function can be characterized as being the unique
solution to the 
Dinamic Programing Principle (DPP)
\[
\left\{
\begin{array}{ll}
\displaystyle u^\eps (x,t) = \inf_{{dim}(S)=j} \sup_{v\in S, |v|=1}
\left\{ \frac{1}{2} u^\eps (x + \eps v,t-\frac{\eps^2}{2}) + \frac{1}{2} u^\eps (x - \eps v,t-\frac{\eps^2}{2})
\right\}  & x \in \Omega, \, t>0, \\[10pt]
u^\eps (x,t) = h(x,t)  & x \not\in \Omega, \text{ or }  t \leq 0.
\end{array}
\right.
\]
Moreover, if $\Omega$ satisfies conditions ($F_j$) and ($F_{N-j+1}$), there exists a function $u \in C(\overline{\Omega} \times [0,+\infty)$ such that
\[
\begin{split}
u^{\eps}\to u \qquad\textrm{ uniformly in}\quad\overline{\Omega} \times [0,T],
\end{split}
\]
as $\eps\to 0$ for every $T>0$.
This limit $u$ is the unique viscosity solution to 
$$
\left\lbrace\begin{array}{ll}
u_t (x,t) - \lambda_j(D^2 u (x,t)) = 0, & \text{in } \Omega\times (0,+\infty), \\
u(x,t) = g(x,t), & \text{on } \partial \Omega \times (0,+\infty), \\
u(x,0) = u_0(0), & \text{in } \Omega.
\end{array}\right.
$$
\end{theorem}

Once we proved existence and uniqueness of solutions, we focus on their asymptotic behaviour as $t\to \infty$.
We restrict our attention to the case where the boundary datum does not depend on $t$, that is,
\begin{equation} \label{eq.ppp}
\left\lbrace\begin{array}{ll}
u_t (x,t) - \lambda_j(D^2 u(x,t)) = 0, & \text{in } \Omega\times (0,+\infty), \\
u(x,t) = g(x), & \text{on } \partial \Omega \times (0,+\infty), \\
u(x,0) = u_0(x), & \text{in } \Omega, 
\end{array}\right.
\end{equation}
where $u_0$ is a continuous function defined on $\overline{\Omega}$ and $g = u_0 |_{\partial\Omega}$.

Using PDE techniques, a comparison argument with super and sub solutions constructed using an associated eigenvalue problem,
we can show that $u(x,t)$ converges exponentially fast to the stationary solution. 
In the special case of $j=1$ (or $j=N$) this result provides us with 
an approximation of the convex envelope (or the concave envelope) of a boundary datum by solutions
to a parabolic problem.

\begin{theorem} \label{teo.comp.asymp.intro}
Let $\Omega\subset \mathbb{R}^N$ be an open bounded domain,
and let $u_0$ be a continuous function defined on $\overline{\Omega}$ and $g = u_0 |_{\partial\Omega}$.
Then, there exist two positive constants, $\mu >0$ (that depends on $\Omega$) and $C$ (depending on the initial condition $u_0$) such that
the unique viscosity solution $u$ of \eqref{eq.ppp} verifies
\begin{equation} \label{cotas.asimp.intro}
\| u(\cdot,t) - z(\cdot) \|_{\infty} \leq C e^{-\mu t} ,
\end{equation}
where $z$ is the unique viscosity solution of \eqref{convex envelope equation}.
\end{theorem}

In addition, we also describe an interesting behavior of the solutions.
Let us present our ideas in the simplest case and consider the special case $j=1$ with $g\equiv 0$, that is we deal with the problem 
$$
\left\lbrace\begin{array}{ll}
u_t (x,t) - \lambda_1(D^2 u (x,t) )=0, & \text{in } \Omega\times (0,+\infty), \\
u (x,t) = 0, & \text{on } \partial\Omega \times (0,+\infty), \\
u (x,0) = u_0, & \text{in } \Omega.
\end{array}\right.
$$
Notice that in this case $z\equiv 0$ and from Theorem \ref{teo.comp.asymp.intro} we have that 
$u(x,t) \to 0$ exponentially fast, $- C e^{-\mu t } \leq u(x,t) \leq C e^{-\mu t}$. In this scenario we can 
improve the upper bound. We show that there exists a finite time 
$T>0$ depending only on $\Omega$, such that the solution satisfies $u(x,t)\leq 0$, for any $t>T$. 
This is a consequence of the fact that the eigenvalue problem
\begin{equation}\label{lambda 1 eigenvalue eq}
-\lambda_1 (D^2\varphi (x)) = \mu \varphi (x), \quad  \text{in } \Omega, 
\end{equation}
admits a positive solution for any $\mu>0$ whenever $\Omega$ is bounded.
In particular, this result says that, for $g\equiv 0$ and $j=1$, there exists $T>0$ such that the solution of \eqref{convex envelope evolution} is below the convex envelope of $g$ in $\Omega$ for any time beyond $T$.
In fact, the same situation occurs when $g$ is an affine function (we just apply the same argument to $\tilde{u} = u - g$).
When we consider $j=N$ and an affine function $g$, we have the analogous behavior, i.e. there exists $T>0$ such that the solution of \eqref{convex envelope evolution} is above the concave envelope of $g$ in $\Omega$ for any time beyond $T$. However, the situation is different when one considers $1<j<N$. 
In this case, equation \eqref{lambda 1 eigenvalue eq} admits a positive and a negative solution for any $\mu>0$,
 and therefore, it is possible to prove the existence of $T>0$, depending only on $\Omega$, such that $u(x,t) = z(x)$, for any $t>T$, where $z$ is the solution of \eqref{convex envelope equation}. We sum up these results in the following theorem.
 
\begin{theorem} \label{teo.comp.cur.intro}
Let $\Omega\subset \mathbb{R}^N$ be an open bounded domain. 
Let $g$ be the restriction of an affine function to $\partial\Omega$ and $u_0$ a continuous function in $\Omega$. If $u(x,t)$ is the viscous solution of \eqref{eq.ppp} and $z(x)$ is the 
affine function (that turns out to be the viscous solution of \eqref{convex envelope equation}), then there exists $T>0$, depending only on $\Omega$, such that
\begin{enumerate}
\item If $j=1$, then $u(x,t)\leq z(x)$, for any $t>T$.
\item If $j=N$, then $u(x,t)\geq z(x)$, for any $t>T$.
\item If $1<j<N$, then $u(x,t)=z(x)$, for any $t>T$.
\end{enumerate}
\end{theorem}

Notice that (iii) says that for $1<j<N$ and an affine boundary datum we have convergence to the stationary solution {\it in finite time}.

Although this result implies that for some situations the exponential decay given in Theorem \ref{teo.comp.asymp.intro} is not sharp, 
we will also describe some other situations (with boundary data that are not affine functions) for which the solution $u(x,t)$ does not 
fall below or above the convex or concave envelope in finite time.

In this final part of the paper we also look at the asymptotic behaviour of the values of the game described above and show that
there exists $\mu >0$, a constant depending only on $\Omega$, and $C$ independent of $\eps$ such that
$$
\| u^\eps (\cdot,t) - z^\eps (\cdot)\|_\infty \leq C e^{-\mu t},
$$
being $u^\eps$ the value function for the game and $z^\eps$ a stationary solution to de (DPP).
Note that from here we can provide a different proof (using games) of Theorem \ref{teo.comp.asymp.intro}.
We also provide a new proof of Theorem \ref{teo.comp.cur.intro} using game theoretical arguments.
With these techniques we can obtain a similar result in the case that $g$ coincide with an affine function in a half-space.

Moreover, thanks to the game approach we can show a more bizarre behaviour in a simple configuration of 
the data. Consider the equation $u_t = \lambda_j (D^2u)$. 
Let $\Omega$ be a ball centered at the origin, $\Omega= B_R\subset \mathbb{R}^N$, and call $(x',x'')\in \mathbb{R}^{j}\times\mathbb{R}^{N-j}$. Assume that the boundary datum is given by two affine functions
(for example, take $g(x',x'') = |x''|$, for $ (x',x'')\in \R^N\setminus \Omega$) and the initial condition is strictly positive inside $\Omega$, 
$ u_0 > 0$.  
For this choice of $g$, we have that the stationary solution satisfies $
z(x',x'') = 0 $ in $\Omega\cap \{x''=0\}$.
In this configuration of the data we have that for every point $x_0$ in the segment $\{x''=0\}\cap \Omega$, we have
$$
u(x_0,t) >0 = z(x_0)
$$
for every $t>0$. However, for any point $x_0$ outside the segment $\{x''=0\}\cap \Omega$, there exists a finite time
$t_0 $ (depending on $x_0$) such that 
$$
u(x_0,t) = z(x_0)
$$
for every $t>t_0$.

That is, the solution to the evolution problem eventually coincides with the stationary solution outside the segment $\{x''=0\}\cap \Omega$, 
but this fact does not happens on the segment.

\medskip

The paper is organized as follows: in section  \ref{sect-games} we collect some preliminaries
and we use the game theoretical approach to obtain existence of solutions; finally, in section
\ref{sect-asymp} we deal with the asymptotic behaviour of the solutions.

\section{Games} \label{sect-games}

\subsection{Preliminaries on viscosity solutions and a comparison principle} \label{sect:preliminaries}

We begin by stating the usual definition of a viscosity solution to \eqref{convex envelope evolution}.

\begin{definition} \label{def.sol.viscosa}
A function  $u:\Omega_T := \Omega \times (0,+\infty) \to \mathbb{R}$  verifies
$$
u_t-\lambda_j ( D^2 u )  = 0
$$
\emph{in the viscosity sense} if the lower and upper semicontinuous envelopes of $u$ given respectively by
$$
\begin{array}{l}
u_\ast (x,t) = \displaystyle\sup_{r>0} \ \inf \{ u(y,s); \quad y\in B_r(x),\ |s-t|<r\}, \\
 \noalign{\vskip 1mm}
u^\ast (x,t) = \displaystyle\inf_{r>0} \ \sup \{ u(y,s); \quad y\in B_r(x),\ |s-t|<r\},
\end{array}
$$
satisfy
\begin{enumerate}
\item for every $\phi\in C^{2}(\Omega_T)$ such that $u_*-\phi $ has a strict
minimum at the point $(x,t) \in \Omega_T$
with $u_*(x,t)=\phi(x,t)$,
we have
$$
\phi_t (x,t)-\lambda_j ( D^2 \phi  (x,t))  \geq 0.
$$

\item for every $ \psi \in C^{2}(\Omega_T)$ such that $ u^*-\psi $ has a
strict maximum at the point $ (x,t) \in \Omega_T$
with $u^*(x,t)=\psi(x,t)$,
we have
$$
\psi_t (x,t) -\lambda_j ( D^2 \psi (x,t) ) \leq 0.
$$
\end{enumerate}
\end{definition}

From our results we will obtain a solution that is continuous in the whole $\overline{\Omega_T}$ and
hence we can avoid the use of $u^*$ and $u_*$ in what follows.

Comparison holds for our equation, see Theorem 8.2 from \cite{CIL}.
Let $\overline{u}$ be a supersolution, that is, it verifies
\begin{equation}\label{supersol}
\left\lbrace\begin{array}{ll}
\overline{u}_t (x,t) - \lambda_j(D^2_x \overline{u}(x,t)) \geq 0, & \text{in } \Omega\times (0,+\infty), \\
\overline{u}(x,t) \geq g(x,t), & \text{on } \partial \Omega \times (0,+\infty), \\
\overline{u}(x,0) \geq u_0(x), & \text{in } \Omega,
\end{array}\right.
\end{equation}
and $\underline{u}$ be a subsolution, that is,
\begin{equation}\label{ec.subsol}
\left\lbrace\begin{array}{ll}
\underline{u}_t (x,t) - \lambda_j(D^2_x \underline{u}(x,t)) \leq 0, & \text{in } \Omega\times (0,+\infty), \\
\underline{u} (x,t)  \leq g(x,t), & \text{on } \partial \Omega \times (0,+\infty), \\
\underline{u} (x,0)  \leq u_0(x), & \text{in } \Omega.
\end{array}\right.
\end{equation}
Notice that the inequalities $\overline{u}_t (x,t) - \lambda_j(D^2_x \overline{u}(x,t)) \geq 0$ and 
$\underline{u}_t (x,t) - \lambda_j(D^2_x \underline{u}(x,t)) \leq 0$ are understood in the viscosity sense 
(see Definition \ref{def.sol.viscosa}), while the other inequalities (that involve boundary/initial data) are understood 
in a pointwise sense. 

\begin{lemma} \label{lemma.compar} Let $\overline{u}$ and $\underline{u}$ verify \eqref{supersol}
and \eqref{ec.subsol} respectively, then 
$$\overline{u} (x,t) \geq \underline{u} (x,t)$$
for every $(x,t) \in \Omega \times (0, +\infty)$.
\end{lemma}

As an immediate consequence of this result uniqueness of continuous viscosity solutions to our problem 
\eqref{convex envelope evolution} follows.

\subsection{Existence using games} \label{sect-games.subsect}

Let $\Omega \subset\R^N$ be a bounded open set and $T>0$.
We define $\Omega_T=\Omega\times (0,T]$. Two
values, $\eps>0$ and $j\in\{1,\dots,N\}$, are given.
The game under consideration is a two-player zero-sum game that is played in the domain $\Omega_T$. 
Initially, a token is placed at some point $(x_0,t_0)\in \Omega_T$. 
Player~I chooses a subspace $S$ of dimension $j$
and then Player~II chooses one unitary vector, $v$, in the subspace $S$.
Then the position of the token is moved to $(x_0\pm \eps v,t_0-\frac{\eps^2}{2})$ with equal probabilities. 
After the first round, the game continues from $(x_1,t_1)$ according to the same rules.
This procedure yields a sequence of game states
\[
(x_0,t_0),(x_1,t_1),\ldots
\]
 where every $x_k$ is a random variable.
The game ends when the token leaves $\Omega_T$, at this point the 
token will be in the parabolic boundary strip of width $\eps$ given by
\[
\Gamma^\eps_T=
\left(\Gamma^\eps\times\left[-\frac{\eps^2}{2},T\right]\right)
\cup
\left(\Omega\times\left[-\frac{\eps^2}{2},0\right]\right)
\]
where
\[
\begin{split}\Gamma^\eps=
\{x\in {\mathbb{R}}^N \setminus \Omega \,:\,\dist(x,\partial \Omega )\leq \eps\}.
\end{split}
\]
We denote by $(x_\tau,t_\tau )\in \Gamma^\eps_T$ the first point in the
sequence of game states that lies in $\Gamma^\eps_T$, so that $\tau$
refers to the first time we hit $\Gamma^\eps_T$.
At this time the game ends with the final payoff given by $h(x_\tau,t_\tau)$, where
$h:\Gamma^\eps_T \to\R$ is a given continuous function that we call \emph{payoff function}.
Player~I earns $-h(x_\tau,t_\tau)$ while Player~II
earns $h(x_\tau,t_\tau)$ (recall that this game is a zero-sum game). For our purposes we choose
\begin{equation}\label{h def}
h(x,t) = \left\{ 
\begin{array}{ll}
g(x,t), \qquad  & x\in\partial\Omega, t >  0, \\
u_0(x) , \qquad & x \in\Omega, t = 0.
\end{array}
\right.
\end{equation}

A strategy $S_\I$ for Player~I, the player seeking to minimize the final payoff, is a function defined on the
partial histories that at every step of the game gives a $j-$dimensional subspace $S$ 
\[
S_\I{\left(t_0,x_0,x_1,\ldots,x_k\right)}=S\in  Gr(j,\R^N).
\]
A strategy $S_\II$ for Player~II, who seeks to maximize the final payoff, is a function defined on the
partial histories that at every step of the game gives a unitary vector in a prescribed $j-$dimensional subspace $S$ 
\[
S_\II{\left(t_0,x_0,x_1,\ldots,x_k,S\right)}=v\in S.
\]

When the two players fix their strategies $S_I$ and $S_{II}$  we can compute the expected outcome as follows:
Given the sequence $(x_0,t_0),(x_1,t_1),\ldots,(x_k,t_k)$  
 in $\Omega_T$,
the next game position is distributed according to the probability
\[
\begin{split}
\pi_{S_\I,S_\II}&((x_0,t_0),(x_1,t_1),\ldots,(x_k,t_k),{A})
= \frac12 \delta_{(x_k+\eps v, t_k-\frac{\eps^2}{2})}(A)+
\frac12 \delta_{(x_k-\eps v, t_k-\frac{\eps^2}{2})}(A),
\end{split}
\]
for all $A\subset \Omega_T \cup \Gamma_T^\eps$, where $v=S_\II{\left(t_0,x_0,x_1,\ldots,x_k,S_\I{\left(t_0,x_0,x_1,\ldots,x_k\right)}\right)}$.
By using the one step transition probabilities and Kolmogorov's extension theorem, we can build a
probability measure $\mathbb{P}^{x_0,t_0}_{S_\I,S_\II}$ on the
game sequences for which the initial position is $(x_0,t_0)$, that we call $H^\infty$. The expected payoff, when starting from $(x_0,t_0)$ and
using the strategies $S_\I,S_\II$, is then computed according to this probability as
\begin{equation}
\label{eq:expectation}
\mathbb{E}_{S_{\I},S_\II}^{x_0,t_0}\left[h(x_\tau,t_\tau)\right]=\int_{H^\infty} h(x_\tau,t_\tau) \ud
\mathbb{P}^{x_0,t_0}_{S_\I,S_\II}.
\end{equation}

The \emph{value of the game for Player I} is define as
\[
u^\eps_\I(x_0,t_0)=\inf_{S_\I}\sup_{S_{\II}}\,
\mathbb{E}_{S_{\I},S_\II}^{x_0,t_0}\left[h(x_\tau,t_\tau)\right]
\]
while the \emph{value of the game for Player II} as
\[
u^\eps_\II(x_0,t_0)=\sup_{S_{\II}}\inf_{S_\I}\,
\mathbb{E}_{S_{\I},S_\II}^{x_0,t_0}\left[h(x_\tau,t_\tau)\right].
\]
Intuitively, the values $u_\I^\eps(x_0,t_0)$ and $u_\II^\eps(x_0,t_0)$ are the best
expected outcomes each player can expect when the game starts at
$(x_0,t_0)$. 
If these two values coincide, $u^\eps_\I= u^\eps_\II$, we say that the game has a value.

Let us observe that the game ends after at most a finite number of steps, in fact, we have
$$\tau \leq { \left\lceil \frac{2T}{\eps^2}\right\rceil }.$$
Hence, the expected value computed in \eqref{eq:expectation} is well defined.

To see that the game has a value, we can consider $u^\eps$, a function that satisfies the 
Dynamic Programing Principle (DPP) associated with this game, that is given by
\begin{equation*}
\left\{
\begin{array}{ll}
\displaystyle u^\eps (x,t) = 
\inf_{{dim}(S)=j} \sup_{v\in S, |v|=1}
\left\{ \frac{1}{2} u^\eps \Big(x + \eps v,t-\frac{\eps^2}{2}\Big) + \frac{1}{2} u^\eps \Big(x - \eps v,t-\frac{\eps^2}{2}\Big)
\right\}  & (x,t) \in \Omega_T, \\[10pt]
u^\eps (x,t) = h(x,t)  & (x,t) \not\in \Omega_T.
\end{array}
\right.
\end{equation*}
The existence of such a function can be seen defining the function backwards in time.
In fact, given $h(x,t)$ we can compute $u^\eps (x,t)$ for $0<t<\frac{\eps^2}{2}$ using the (DPP) and then continue
with $u^\eps$ for $\frac{\eps^2}{2} \leq t < 2\frac{\eps^2}{2}$, etc.

Now, we want to prove that a function that verifies the (DPP) $u^\eps$, is in fact the value of the game, that is, it holds that $u^\eps=u^\eps_\I= u^\eps_\II$.
We know that $u^\eps_\II\leq u^\eps_\I$, to obtain the equality, we will show that $u^\eps \leq u^\eps_\II$ and $u^\eps_\I\leq u^\eps$.

Given $u^\eps$ a function that verifies the (DPP) and $\eta>0$, we can consider the strategy $S_\II^0$ for Player~II that at every step almost maximize 
$u^\eps (x_k + \eps v,t_k-\frac{\eps^2}{2}) + u^\eps (x_k - \eps v,t_k-\frac{\eps^2}{2})$,
that is
\[
S_\II^0{\left(t_0,x_0,x_1,\ldots,x_k,S\right)}=w\in S
\]
such that
\[
\begin{split}
\frac{1}{2} u^\eps (x_k + \eps w,t_k-\frac{\eps^2}{2}) + \frac{1}{2} u^\eps (x_k - \eps w,t_k-\frac{\eps^2}{2}) \geq
\qquad\qquad\qquad\qquad\qquad\qquad\\
\sup_{v\in S, |v|=1}
\left\{ \frac{1}{2} u^\eps (x_k + \eps v,t_k-\frac{\eps^2}{2}) + \frac{1}{2} u^\eps (x_k - \eps v,t_k-\frac{\eps^2}{2}) \right\}
-\eta 2^{-(k+1)} 
\end{split}
\]

We have
\[
\begin{split}
&\mathbb{E}_{S_\I, S^0_\II}^{x_0,t_0}[u^\eps(x_{k+1},t_{k+1})-\eta 2^{-(k+1)}|\, x_0,\ldots,x_k]
\\
&\qquad \qquad \geq
\inf_{S , {dim}(S)=j} \sup_{v\in S, |v|=1}
\left\{ \frac{1}{2} u^\eps (x_k + \eps v,t_k-\frac{\eps^2}{2}) + \frac{1}{2} u^\eps (x_k - \eps v,t_k-\frac{\eps^2}{2})
\right\}
\\
& \qquad\qquad\qquad\qquad\qquad\qquad\qquad\qquad -\eta 2^{-(k+1)}-\eta 2^{-(k+1)}
\\
&\qquad \qquad \geq u^\eps(x_k,t_k)-\eta 2^{-k},
\end{split}
\]
where we have estimated the strategy of Player I by $\inf$ and used that $u^\eps$ satisfies the (DPP).
Thus
\[
M_k=u^\eps(x_k,t_k)-\eta2^{-k} 
\]
is a submartingale. Now, we have
\[
\begin{split}
u^\eps_\II(x_0,t_0) 
&=\sup_{S_\II}\inf_{S_{\I}}\, \mathbb{E}_{S_{\I},S_\II}^{x_0,t_0}\left[h(x_\tau,t_\tau)\right]\\
&\geq\inf_{S_{\I}}\, \mathbb{E}_{S_{\I},S^0_\II}^{x_0,t_0}\left[h(x_\tau,t_\tau)\right]\\
&\geq \inf_{S_\I} \mathbb{E}_{S_{\I},S^0_\II}^{x_0,t_0}[M_{\tau}]\\
&\geq \inf_{S_\I}\mathbb{E}_{S_{\I},S^0_\II}^{x_0,t_0}[M_0]=u^\eps(x_0,t_0)-\eta,
\end{split}
\]
where we used the optional stopping theorem for $M_{k}$.
Since $\eta$ is arbitrary small, this proves that $u^\eps_\II \geq u^\eps$.
Analogously, we can consider a strategy $S_1^0$ for Player~I to prove that $u^\eps\geq u^\eps_\I$. 
This shows that the game has a value that can be characterized as the solution to the (DPP).

Our next aim now is to pass to the limit in the values of the game
$$
u^\eps \to u
$$
as $\eps\to 0$ and obtain in this limit process a viscosity solution to \eqref{convex envelope evolution}.

We will use the following Arzela-Ascoli type lemma, to obtain a convergent subsequence $u^\eps \to u$. 
For its proof we refer to Lemma~4.2 from \cite{MPRb}.

\begin{lemma}\label{lem.ascoli.arzela}
Let $\{u^\eps : \overline{\Omega} \times [0,T]\to \R,\ \eps>0\}$ be a set of functions such that
\begin{enumerate}
\item there exists $C>0$ such that $\abs{u^\eps (x,t)}<C$ for every $\eps >0$ and every $(x,t) \in \overline{\Omega}\times [0,T]$,
\item given $\eta>0$ there are constants $r_0$ and $\eps_0$ such that for every $\eps < \eps_0$ and any $x, y \in \overline{\Omega}$ with $|x - y | < r_0 $ and for every $t,s\in [0,T]$ with $|t-s| < r_0$ it holds
$$
|u^\eps (x,t) - u^\eps (y,s)| < \eta.
$$
\end{enumerate}
Then, there exists  a uniformly continuous function 
$u:\overline{\Omega} \times [0,T] \to \R$ and a subsequence still denoted by $\{u^\eps \}$ such that
\[
\begin{split}
u^{\eps}\to u \qquad\textrm{ uniformly in}\quad\overline{\Omega} \times [0,T],
\end{split}
\]
as $\eps\to 0$.
\end{lemma}

So, our goal now is to show that the family $u^\eps$ satisfies the hypotheses of the previous lemma.
First, let us observe that
$$
\min h \leq u^\eps (x,t) \leq \max h
$$
for every $(x,t) \in \overline{\Omega} \times [0,T]$.
To prove that $u^\eps$ satisfies the second condition in Lemma \ref{lem.ascoli.arzela} 
we will have to make some geometric assumptions on the domain.
As in \cite{BlancRossi}, given $y\in\partial\Omega$ we assume that there exists
$r>0$ such that for every $\delta>0$ there exists $T \subset\R^N$ a subspace of dimension $j$, $w\in\R^N$ of norm 1,
$\lambda>0$ and $\theta>0$ such that
\begin{equation}
\tag{$F_j$}
\{x\in\Omega\cap B_r(y)\cap T_\lambda: \langle w,x-y\rangle<\theta\}\subset B_\delta(y)
\end{equation}
where 
\[
T_\lambda=\{x\in\R^N: d(x-y,T)<\lambda\}.
\]
For our game with a fixed $j$ we will assume that $\Omega$ satisfies both ($F_j$) and ($F_{N-j+1}$).
As we mentioned in the introduction, observe that every strictly convex domain verifies ($F_j$) for any $1\leq j
\leq N$.

The key point to obtain the asymptotic equicontinuity required in the second condition in Lemma
\ref{lem.ascoli.arzela} is to obtain the bound for $(x,t)\in \Omega_T$ and $(y,s) \in \Gamma^\eps_T$.
For the case $(x,t),(y,s) \in \Gamma^\eps_T$ the bound follows from the uniform continuity of $h$ in $\Gamma^\eps_T$.
For the case $(x,t),(y,s)\in \Omega_T$ we argue as follows.
We fix the strategies $S_\I, S_\II$ for the game starting at $(x,t)$.
We define a virtual game starting at $(y,s)$ using the same random steps as the game starting at $(x,t)$.
Furthermore, the players adopt their strategies $S_\I, S_\II$ from the game starting at $(x,t)$, that is,
when the game position is $(y_k,s_k)$ a player make the choices that would have taken at $(x_k,t_k)$ while playing the game starting at $(x,t)$.
We proceed in this way until for the first time one of the positions leave the parabolic domain, that is, until $(x_k,t_k) \in \Gamma_T^\eps$ 
or $(y_k,s_k) \in \Gamma^\eps_T$.
At that point we have $|(x_k,t_k)-(y_k,s_k)| =|(x,t)-(y,s)|$, and the desired estimate follow from the one for 
for $x_k, y_k \in\Gamma_\eps$ (in the case that both positions leave the domain at the same turn, $k$)
or $x_k \in \Omega$, $y_k\in \Gamma_\eps$ (if only one have leaved the domain).

Thus, we can concentrate on the case $(x,t)\in \Omega_T$ and $(y,s) \in \Gamma^\eps_T$.
We can assume that $(y,s)\in\partial_P\Omega_T$.
If we have the bound for those points we can obtain a bound for a point $(y,s) \in \Gamma^\eps_T$ 
just by considering $(z,u)\in\partial_P\Omega_T$ close to $(x,t)$ and $(y,s)$.
If $s<0$, we can consider the point $(x,0)$ and for $y\not \in\Omega$ we can consider $(z,t)$ with $z\in\partial\Omega$ a point in the line segment  that joins $x$ and $y$.

Hence, we have to handle two cases.
In the first one we have to prove that $|u^\eps(x,t)-u^\eps(x,0)|<\eta$ for $x\in\Omega$ and $0<t<r_0$.
In the second one we have to prove that $|u^\eps(x,t)-u^\eps(y,t)|<\eta$ for $x\in\Omega$, $y\in \partial\Omega$ such that $|x-y|<r_0$ and $0<t\leq T$.

In the first case we have
$$
u^\eps (x,0) = u_0(x),
$$
we have to show that the game starting at $(x,t)$ will not end too far a way from $(x,0)$.
We have $-\frac{\eps^2}{2}< t_\tau<t$, so we have to obtain a bound for $|x-x_\tau|$.
To this end we consider $M_k=|x_k-x|^2-\eps^2k$. 
We have
\begin{equation}\label{martingale}
\begin{split}
& \mathbb{E}^{x,t}_{S_I,S_{II}}[|x_{k+1}-x|^2-\eps^2(k+1)|x,x_1,\dots,x_k] \\
&\qquad =\frac{|x_k+\eps v_k-x|^2+|x_k-\eps v_k-x|^2}{2}-\eps^2(k+1)\\
& \qquad =|x_k-x|^2+\eps^2|v_k|^2-\eps^2(k+1) \\
& \qquad  =M_k.
\end{split}
\end{equation}
Hence, $M_k$ is a martingale.
By applying the optional stopping theorem, we obtain
\begin{equation}\label{stoppingbound}
\mathbb{E}^{x,t}_{S_I,S_{II}}[|x_\tau-x|^2]
=\eps^2\mathbb{E}^{x,t}_{S_I,S_{II}}[\tau]
\leq \eps^2 \left\lceil \frac{2t}{\eps^2}\right\rceil
\leq \eps^2 \left\lceil \frac{2r_0}{\eps^2}\right\rceil
\leq \eps^2+2r_0
\leq \eps_0^2+2r_0
\end{equation}

Hence, using
\[
\mathbb{E}^{x,t}_{S_I,S_{II}}[|x_\tau-x|^2]\geq 
\mathbb{P} (|(x_\tau,t_\tau)-(x,0)|\geq \delta )\delta^2,
\]
we obtain
\[
\mathbb{P} (|(x_\tau,t_\tau)-(x,0)|\geq \delta)\leq \frac{\eps_0^2+2r_0}{\delta^2}.
\]
With this bound, we can obtain the desired result as follows:
\begin{equation}\label{diffbound}
\begin{split}
|u_\eps (x,t) - h(x,0)| 
&\leq \mathbb{P}(|(x_\tau,t_\tau)-(x,0)|< \delta ) \times \sup_{(x_\tau,t_\tau)\in B_\delta(x,0)}|h (x_\tau,t_\tau) - h(x,0)|
\\
& \qquad \qquad  + \mathbb{P}(|(x_\tau,t_\tau)-(x,0)|\geq \delta)) 2 \max |h| \\
&\leq \sup_{(x_\tau,t_\tau)\in B_\delta(x,0)}|h (x_\tau,t_\tau) - h(x,0)| +\frac{(\eps_0^2+2r_0)2 \max |h|}{\delta^2} <\eta
\end{split}
\end{equation}
if $\delta$, $\eps_0$ and $r_0$ are small enough.

Now we move to the second case, we have $u^\eps (y,s) = g(y,s)$,
Here, we need to make the geometric assumptions 
($F_j$) and ($F_{N-j+1}$) on $\partial \Omega$. In this parabolic game we have an extra difficulty compared with the elliptic case
treated in \cite{BlancRossi},
we have to make an extra effort to bound the amount of time that it takes for the game to end.

We start with the case $j=1$, in this case we assume ($F_N$). This condition reads as follows:
For every $y\in\partial\Omega$ we assume that there exists
$r>0$ such that for every $\delta>0$ there exists $w\in\R^N$ of norm 1 and $\theta>0$ such that
\begin{equation} \label{cond.geoN}
\{x\in\Omega\cap B_r(y): \langle w,x-y\rangle<\theta\}\subset B_\delta(y).
\end{equation}

Let us observe that for any possible choice of the direction $v$ at every step
we have that the projection of the position of the game, $x_n$, in the direction of a fixed unitary vector $w$, that is,  
$$
\left\langle x_n-y, w\right\rangle,
$$
is a martingale.
We fix $r>0$ and consider $\tilde\tau$, the first time $x$ leaves $\Omega$ or $B_r(y)$.
Hence
\begin{equation}\label{boundE}
\mathbb{E} \left\langle x_{\tilde\tau}-y, w \right\rangle \leq \left\langle x-y, w \right\rangle \leq d(x,y)  < r_0.
\end{equation}
We consider the vector $w$ given by the geometric assumption on $\Omega$, we have that $$\left\langle x_n-y, w\right\rangle \geq -\eps.$$
Therefore, \eqref{boundE} implies that
$$
\mathbb{P} \left( \left\langle x_{\tilde\tau}-y, w \right\rangle  > r_0^{1/2} \right) r_0^{1/2} -
\left(1-\mathbb{P} \left(  \left\langle x_{\tilde\tau}-y, w \right\rangle  > r_0^{1/2} \right) \right)\eps < r_0.
$$
Hence, we have (for every $\eps>\eps_0$ small enough)
$$
\mathbb{P} \left( \left\langle x_{\tilde\tau}-y, w \right\rangle  > r_0^{1/2} \right)  <  2 r_0^{1/2}.
$$
Then, \eqref{cond.geoN} implies that given $\delta>0$ we can conclude that
$$
\mathbb{P} ( d ( x_{\tilde\tau}, y)  > \delta )  <  2  r_0^{1/2}.
$$
by taking $r_0$ small enough and an appropriate $w$.

Hence, $d(x_{\tilde\tau}, y)\leq \delta$ with probability close to one, and in this case the point $x_{\tilde\tau}$ is actually the point where the process 
has leaved $\Omega$, that is $\tilde\tau=\tau$.
When $d(x_\tau, y)\leq \delta$, by the same martingale argument used in \eqref{stoppingbound}, we obtain
\[
\mathbb{E}[t-t_\tau]=
\mathbb{E}\left[\frac{\eps^2}{2}\tau\right]=
\frac{\mathbb{E}[|x_\tau-x|^2]}{2}\leq \frac{\delta^2}{2} .
\]
Hence,
\[
\mathbb{P}(t-t_\tau>\delta)\leq\frac{\delta}{2}
\]
and the bound follows as in \eqref{diffbound}.

In the general case, for any value of $j$, we can proceed in the same way.
In order to be able to use condition $(F_j)$, we have to argue that the points $x_n$ involved in our argument belong to $T_\lambda$.
For $r_0<\lambda$ we have that $x\in T_\lambda$, so if we ensure that at every move $v\in T$ we will have that the game sequence will be contained in $x+T\subset T_\lambda$.

Recall that here we are assuming that both ($F_j$) and ($F_{N-j+1}$) are satisfied.
We can separate the argument into two parts. We will prove on the one hand that $u_\eps (x,t) - g(y,s)<\eta$ 
and on the other that $g(y,s)-u_\eps (x,t)<\eta$.
For the first inequality we can make choices for the strategy for Player~I, and for the second one we can do the same 
for strategies of Player~II. 

Since $\Omega$ satisfies condition ($F_j$), Player~I can make sure that at every move the vector $v$ belongs to $T$ by selecting $S=T$.
This proves the upper bound $u_\eps (x,t) - g(y,s)<\eta$.
On the other hand, since $\Omega$ satisfies ($F_{N-j+1}$), Player~II will be able to select $v$ in a space $S$ of dimension $j$ and hence he can always choose $v\in S\cap T$ since 
\[
\dim(T)+\dim(S)=N-j+1+j=N+1>N.
\]
This shows the lower bound $g(y,s)-u_\eps (x,t)<\eta$.

We have shown that the hypotheses of the Arzela-Ascoli type lemma, Lemma \ref{lem.ascoli.arzela}, are satisfied.
Hence we have obtained uniform convergence of a subsequence of $u^\eps$.

\begin{lemma} Let $\Omega$ be a bounded domain in $\mathbb{R}^N$ satisfying conditions ($F_j$) and ($F_{N-j+1}$).
Then there exists a subsequence of $u^\eps$ that converges uniformly.
That is,
\[
u^{\eps_j} \to u, \qquad \mbox{ as } \eps_j \to 0,
\]
uniformly in $\overline{\Omega}\times [0,T]$, where $u$ is a uniformly continuous function.
\end{lemma}

Now, let us prove that any possible uniform limit of $u^\eps$ is a viscosity solution to
the limit PDE problem. This result shows existence of a continuous up to the boundary solution
defined in $\overline{\Omega} \times [0,T]$
for every $T>0$. Uniqueness of this viscosity solution follows from the comparison principle stated in Lemma 
\ref{lemma.compar}.

\begin{theorem}
Let $u$ be a uniform limit of the values of the game $u^\eps$. Then $u$ is a viscosity
solution to \eqref{convex envelope evolution}.
\end{theorem}

\begin{proof}
First, we observe that since $u^\eps =g$ on $\partial \Omega \times (0,T)$ and $u^\eps(x,0) = u_0(x)$ for $x\in\Omega$, we obtain,
from the uniform convergence, that $u =g$ on $\partial \Omega \times (0,T)$ and $u(x,0) = u_0(x)$ for $x\in\Omega$.
Also, notice that Lemma \ref{lem.ascoli.arzela} gives that the limit function is continuous. 

To check that $u$ is a viscosity supersolution to $\lambda_j(D^2 u)  = 0$ in $\Omega$, 
let $\phi\in C^{2}{ (\Omega_T)}$ be such that $ u-\phi $ has a strict
minimum at the point $(x,t) \in \Omega_T$  with $u(x,t)=
\phi(x,t)$. We need to check that
$$
\phi_t (x,t) - \lambda_j ( D^2 \phi (x,t) )  \geq 0.
$$
As $u^\eps \to u$ uniformly in $\overline{\Omega}\times {[0,T]}$ we have the existence of two sequences
$x_\eps$, $t_\eps$ such that $x_\eps \to x$, $t_\eps \to t$ as $\eps \to 0$ and 
$$
u^\eps (z,s) - \phi (z,s) \geq u^\eps (x_\eps,t_\eps) - \phi (x_\eps,t_\eps) - \eps^3
$$
(remark that $u^\eps$ is not continuous in general). 
Since $u^\eps$ is a solution to
$$
u^\eps (x,t) =
\inf_{{dim}(S)=j} \sup_{v\in S, |v|=1}
\left\{ \frac{1}{2} u^\eps \Big(x + \eps v,t-\frac{\eps^2}{2}\Big) + \frac{1}{2} u^\eps \Big(x - \eps v,t-\frac{\eps^2}{2}\Big)
\right\} 
$$
we obtain that $\phi$ verifies 
\begin{equation} \label{ttt}
\begin{array}{l}
\phi(x_\eps ,t_\eps) - \phi \Big(x_\eps ,t_\eps-\frac{\eps^2}{2}\Big)  \\[10pt]
\geq \displaystyle\inf_{{dim}(S)=j} \sup_{v\in S, |v|=1}
\left\{ \frac{1}{2} \phi \Big(x_\eps + \eps v, t_\eps-\frac{\eps^2}{2}\Big) + \frac{1}{2} \phi \Big(x_\eps - \eps v, t_\eps-\frac{\eps^2}{2}\Big)
- \phi \Big(x_\eps, t_\eps-\frac{\eps^2}{2}\Big)
\right\} - \eps^3.
\end{array}
\end{equation}

Now, consider the second order Taylor
expansion of $\phi$ (to simplify the notation we omit the dependence of $t$ here)
\[
\phi(y)=\phi(x)
+\nabla\phi(x)\cdot(y-x)
+\frac12\langle D^2\phi(x)(y-x),(y-x)\rangle
+o(|y-x|^2)
\]
as $|y-x|\rightarrow 0$. Hence, we have
\[
\phi(x+\eps v)=\phi(x)+\eps \nabla\phi(x)\cdot v
+\eps^2 \frac12\langle D^2\phi(x)v,v\rangle+o(\eps^2)
\]
and
\[
\phi(x- \eps v)=\phi(x) - \eps \nabla\phi(x)\cdot v
+\eps^2 \frac12\langle D^2\phi(x)v,v\rangle+o(\eps^2).
\]
Using these expansions we get
$$
\frac{1}{2} \phi (x_\eps + \eps v) + \frac{1}{2} \phi (x_\eps - \eps v)
- \phi (x_\eps) = \frac{\eps^2}2 \langle D^2 \phi (x_\eps)v, v \rangle + o(\eps^2).
$$

Plugging this into \eqref{ttt} and dividing by $\eps^2/2$, we obtain
$$
\dfrac{ \phi (x_\eps,t_\eps) - \phi \left(x_\eps, t_\eps - \dfrac{\eps^2}{2}\right) }{\eps^2/2}
\geq  \inf_{{dim}(S)=j} \sup_{v\in S, |v|=1}
\left\{ \langle D^2 \phi (x_\eps,t_\eps - \eps^2/2) v, v \rangle + 2\dfrac{o(\eps^2)}{\eps^2} \right\} - 2 \eps.
$$

Therefore, passing to the limit as $\eps \to 0$ in \eqref{ttt} we conclude that
$$
\phi_t (x,t) \geq  \inf_{{dim}(S)=j} \sup_{v\in S, |v|=1}
\Big\{ \langle D^2 \phi (x,t) v, v \rangle 
\Big\}.
$$
which is equivalent to
$$
\phi_t (x,t)  \geq \lambda_j (D^2 \phi (x)) 
$$
as we wanted to prove. 

When we consider a smooth function $\psi$ that touches $u$ from above,
we can obtain the reverse inequality in a similar way.
\end{proof}

\begin{remark}{\rm
From the uniqueness of viscosity solutions to the limit problem (recall that a comparison principle holds) we obtain that the convergence of the whole family $u^\eps$.
That is,
\[
u^\eps \to u
\]
uniformly as $\eps\to 0$ (not only along subsequences).
Hence, we have completed the proof of Theorem~\ref{teo.convergencia.juego}.
}
\end{remark}

\section{Asympitotic behaviour} \label{sect-asymp}

Along this section we restrict our attention to the case where the boundary condition does not depend on the time, that is,
$$
\left\lbrace\begin{array}{ll}
u_t (x,t) - \lambda_j(D^2_x u (x,t)) = 0, & \text{in } \Omega\times (0,+\infty), \\
u(x,t) = g(x), & \text{on } \partial \Omega \times (0,+\infty), \\
u(x,0) = u_0(0), & \text{in } \Omega. 
\end{array}\right.
$$
where $u_0$ is a continuous function defined on $\overline{\Omega}$ and $g = u_0 |_{\partial\Omega}$.

We want to study the asymptotic behaviour as $t\to \infty$ of the solution to this parabolic equation.
We deal with the problem with two different techniques, on the one hand we use pure PDE methods (comparison arguments) 
and on the other hand we use our game theoretical approach.

\subsection{PDE arguments}

We will use the eigenvalue problem associated with $-\lambda_N (D^2u)$.
For every strictly convex domain there is a positive eigenvalue $\mu_1$, with an eigenfunction $\psi_1$
that is  positive inside $\Omega$ and continuous up to the boundary with $\psi_1|_{\partial \Omega} =0$ such that
\begin{equation}\label{autovalor.positivo}
\left\lbrace\begin{array}{ll}
- \lambda_N (D^2 \psi_1) = \mu_1 \psi_1, & \text{in } \Omega, \\
\psi_1 = 0, & \text{on } \partial \Omega.
\end{array}\right.
\end{equation}
This eigenvalue problem was studied in \cite{BirindelliIshii}.
Notice that $\varphi_1 = - \psi_1$ is a negative solution to 
\begin{equation}\label{autovalor.negativo}
\left\lbrace\begin{array}{ll}
- \lambda_1 (D^2 \varphi_1) = \mu_1 \varphi_1, & \text{in } \Omega, \\
\varphi_1 = 0, & \text{on } \partial \Omega.
\end{array}\right.
\end{equation}

We will use the following lemma.

\begin{lemma} \label{lema.matrices}
For any two symmetric matrices $A$, $B$, it holds that
\begin{equation} \label{eq.matrices}
\lambda_1 (A) + \lambda_j (B) \leq \lambda_j (A+B) \leq 
\lambda_N (A) + \lambda_j (B).
\end{equation}
\end{lemma}

\begin{proof}
Given a subspace $S$ of dimension $j$, we have
\[
\sup_{v\in S, |v|=1}\langle Bv, v \rangle+ \inf_{|v|=1}\langle Av, v \rangle 
\leq
\sup_{v\in S, |v|=1}\langle(A+B)v, v \rangle 
\leq
\sup_{v\in S, |v|=1}\langle Bv, v \rangle+ \sup_{|v|=1}\langle Av, v \rangle .
\]
Hence, the first inequality follows from
$$
\begin{array}{l}
\displaystyle
\lambda_j( A+B) 
=
\inf_{{dim}(S)=j} \sup_{v\in S, |v|=1}
\langle(A+B)v, v \rangle 
 \\  \displaystyle \qquad \leq 
\inf_{{dim}(S)=j} \sup_{v\in S, |v|=1}
\langle Bv, v \rangle 
+ \displaystyle \sup_{|v|=1}\langle Av, v \rangle
\\ \displaystyle \qquad = \lambda_N (A) + \lambda_j (B)
\end{array}
$$
and the second one from 
$$
\begin{array}{l}
\displaystyle
\lambda_j( A+B) 
=
\inf_{{dim}(S)=j} \sup_{v\in S, |v|=1}
\langle(A+B)v, v \rangle \\
\displaystyle \qquad
\geq 
\inf_{{dim}(S)=j} \sup_{v\in S, |v|=1}
\langle Bv, v \rangle 
+ \inf_{|v|=1}\langle Av, v \rangle \\
\displaystyle \qquad
= \lambda_1 (A) + \lambda_j (B).
\end{array}
$$
This ends the proof.
\end{proof}

\begin{theorem} \label{teo.comportamiento.asintotico} Let $u_0$ be continuous with $u_0|_{\partial \Omega } =g$.
Let $\psi_R$ and $\varphi_R$ be the eigenfunctions associated with $\mu_R$ the first eigenvalue for 
\eqref{autovalor.positivo} and \eqref{autovalor.negativo}
in a large {strictly convex} domain $\Omega_R$ such that $\Omega \subset \subset \Omega_R$.
 Then,
there exist two positive constants {$C_1, C_2$, depending on the initial condition $u_0$,} such that
\begin{equation} \label{cotas.asimp}
z(x) + C_1 e^{-\mu_R t} \varphi_R (x) \leq u(x,t) \leq 
z(x) + C_2 e^{-\mu_R t} \psi_R (x).
\end{equation}
\end{theorem}

\begin{proof}
We just observe that 
$\underline{u}(x,t)=z(x) + C_1 e^{-\mu_R t} \varphi_R (x)$ with $C_1$ large enough is a subsolution 
to our evolution problem in $\Omega$. In fact, we have
$$
\underline{u}_t (x,t) = -\mu_R C_1 e^{-\mu t} \varphi_R (x)
$$
and
$$
\begin{array}{l}
\displaystyle
\lambda_1 (D^2 \underline{u} (x,t) )= \lambda_1 (D^2 z(x) + C_1 e^{-\mu_R t} D^2 \varphi_R (x))
\\[8pt]
\displaystyle \qquad \geq \lambda_1 (D^2 z(x)) + C_1 e^{-\mu_R t}  \lambda_1(D^2 \varphi_R (x))
= - \mu C_1 e^{-\mu_R t}   \varphi_R (x).
\end{array}
$$
An analogous computation shows that $\overline{u} (x,t) = z(x) + C_2 e^{-\mu_R t} \psi_R (x)$ is a supersolution.

In addition, we have that
$$\underline{u} (x,t) \leq g(x) \leq \overline{u} (x,t), \qquad x\in \partial \Omega, t>0, $$
and for $C_1$, $C_2$ large enough (depending on $u_0$)
$$
\underline{u} (x,0) = z(x) + C_1 \varphi_R (x) \leq u_0(x) \leq \overline{u} (x,0) =
z(x) + C_2 \psi_R (x), \qquad x\in  \Omega. 
$$  
Notice that here we are using that $\varphi_R$ and $\psi_R$ are strictly negative and strictly
positive respectively inside $\Omega_R$.

Finally we apply the comparison principle in $\Omega$ to obtain the desired conclusion
$$
z(x) + C_1 e^{-\mu_R t} \varphi_R (x) \leq u(x,t) \leq 
z(x) + C_2 e^{-\mu_R t} \psi_R (x).
$$
\end{proof}

As an immediate consequence of this result we obtain that solutions to our
evolution problem converge uniformly to the convex envelope of the boundary condition. This proves Theorem \ref{teo.comp.asymp.intro}.

Notice that in the previous result $\mu$ is the first eigenvalue for $-\lambda_N(D^2u)$ in the {\it larger} domain $\Omega_R$.
Now, our aim is to obtain a sharper bound (involving $\mu_1$ the first eigenvalue in $\Omega$).
To this aim we have to assume that $u_0$ is $C^1(\overline{\Omega})$ with $u_0|_{\partial \Omega } =g$
and that the solution $z$ of \eqref{convex envelope equation} is $C^1(\overline{\Omega})$. This regularity of 
the solution of \eqref{convex envelope equation}
up to the boundary is not included in \cite{OS} (there only interior regularity for the convex envelope is shown). Under these 
hypotheses on $u_0$ and $z$ the
difference $u_0-z$ is $C^1(\overline{\Omega})$ and vanishes on $\partial \Omega$.
Notice that we do not know if there is a regularizing effect for our evolution problem. That is, we do not know if
for a smooth boundary datum and a continuous initial condition the solution is smooth in $\overline{\Omega}$ for any positive
time $t$ (as happens with solutions to the heat equation). 

As a previous step in our arguments, we need to show that the eigenfunctions have a "negative normal derivative". 
Notice that the existence of such eigenfunction is proved in \cite{Birin2} for strictly convex domains.
Although this hypothesis is sufficient but not necessary
(see \cite{BirindelliIshii} for construction of eigenfunctions in rectangles),
we shall assume it here since the optimal hypotheses for existence of eigenfunctions are unknown (as far as we know).
In the next two results we need to assume that the domain $\Omega$ has some extra regularity (it has an interior tangent ball at every 
boundary point).

\begin{lemma} \label{lemma.normal.derivative} Assume that $\Omega$ is strictly convex and has an interior tangent ball at every point of its boundary.
Let $\varphi_1$ and $\psi_1$ be the eigenfunctions associated with $\mu_1$ the first eigenvalue for 
\eqref{autovalor.positivo} and \eqref{autovalor.negativo}
in $\Omega$. Assume that they are normalized with $\|\psi\|_\infty = \|\varphi\|_\infty =1$.
Then, there exists $C>0$ such that
$$
\psi_1 (x) \geq C \dist (x, \partial \Omega) \qquad \mbox{and} \qquad 
\varphi_1 (x) \leq - C \dist (x, \partial \Omega),
$$
for $x\in \Omega$.
\end{lemma}

\begin{proof}
Take $x_0\in \partial \Omega$. Let $B_r (y)$ be a ball inside $\Omega$, tangent to 
$\partial \Omega$ at $x_0$.  In $B_{r/2} (y)$ the eigenfunction $\psi_1$ is strictly positive
and then we obtain that there exists a constant $c$ such that
$$
\mu_1 \psi_1(x) \geq c, \qquad x \in B_{r/2} (y).
$$
Now, we take $a(x)$ the solution to
\begin{equation} \label{eq.a}
\left\lbrace\begin{array}{ll}
- \lambda_N (D^2 a(x)) = c \chi_{B_{r/2} (y)}(x), & \text{in } B_{r} (y), \\
a(x) = 0, & \text{on } \partial B_{r} (y).
\end{array}\right.
\end{equation}
This function $a$ is radial $a(x) = a(|x-y|)$ and can be explicitly computed. In fact,
$$
a(x) =\left\lbrace\begin{array}{ll}
c_1 (r-|x-y|) , & \text{in } B_{r} (y) \setminus B_{r/2} (y), \\[8pt]
\displaystyle c_2 - \frac{c}{2} |x-y|^2, & \text{in } B_{r/2} (y)
\end{array}\right.
$$
with $c_1$, $c_2$ such that $c_1 = c r/2$ (continuity of the derivative
at $r/2$) and $c_1 r/2 = c_2 - c/2 (r/2)^2$ (continuity of the function at $r/2$).

To conclude we use the comparison argument for \eqref{eq.a} to obtain that
$$
a(x) \leq \psi_1 (x) \qquad x\in B_{r} (y).
$$
This implies that
$$
\psi_1 (x) \geq C \dist (x, \partial \Omega).
$$
A similar argument shows that
$$
\varphi_1 (x) \leq - C \dist (x, \partial \Omega).
$$
\end{proof}

\begin{theorem} \label{teo.comportamiento.asintotico.mu1} 
Assume that $\Omega$ is strictly convex and has an interior tangent ball at every point of its boundary.
Let $g$ be such that the solution $z$ of \eqref{convex envelope equation} is $C^{1} (\overline{\Omega})$ and let $u_0$ be 
$C^1(\overline{\Omega})$ with $u_0|_{\partial \Omega } =g$ and let $\mu_1$ the first eigenvalue for 
\eqref{autovalor.positivo} and \eqref{autovalor.negativo}
in $\Omega$. Then,
there exist two positive constants (depending on the initial condition $u_0$) such that
\begin{equation} \label{cotas.asimp.22}
z(x) + C_1 e^{-\mu_1 t} \varphi_1 (x) \leq u(x,t) \leq 
z(x) + C_2 e^{-\mu_1 t} \psi_1 (x).
\end{equation}
\end{theorem}

\begin{proof}
We just observe that the arguments used in the proof of Theorem \ref{teo.comportamiento.asintotico}
also work here since we can find two constants $C_1$ and $C_2$ such that
\begin{equation} \label{desig.inic}
 z(x) + C_1 \varphi_1 (x) \leq u_0(x) \leq 
z(x) + C_2 \psi_1 (x), \qquad x\in  \Omega. 
\end{equation} 
Here we are using that $u_0 - z$ is $C^1 (\overline{\Omega})$ with $(u_0 - z)|_{\partial \Omega } =0$
to get that there is a constant $C$ such that
$$
-C \dist (x, \partial \Omega) \leq (u_0 - z)(x) \leq C \dist (x, \partial \Omega),
$$
and observe that from our previous Lemma \ref{lemma.normal.derivative} we obtain 
\eqref{desig.inic}.
\end{proof}

We next give the proof of Theorem \ref{teo.comp.cur.intro},
which is a refined description of the asymptotic behavior of the solution to \eqref{convex envelope evolution}
when the boundary datum $g$ comes from the restriction of an affine function to $\partial\Omega$.
For instance, if we consider the case $j=1$, it shows that there exists a finite time $T>0$ beyond which 
the upper estimate in \eqref{cotas.asimp} can be reduced to $z(x)$, the $\lambda_j-$envelope of $g$ inside $\Omega$.

\begin{proof}[Proof of Theorem \ref{teo.comp.cur.intro}]
We assume that there is an affine function (a plane if we are in the case $N=2$) $\pi$ such that $g=\pi |_{\partial \Omega}$.
In this case the $\lambda_j-$envelope $z$ of $g$ inside $\Omega$ is given by 
$$
z(x) = \pi (x).
$$

Hence, let us consider 
$$
\hat{u} (x,t) = u(x,t) - z(x) = u(x,t) - \pi(x).
$$
This function $\hat{u}$ is the viscosity solution to
\begin{equation}\label{hat.u.eq}
\left\lbrace\begin{array}{ll}
\hat{u}_t - \lambda_j(D^2 \hat{u}) = 0, & \text{in } \Omega\times (0,+\infty), \\
\hat{u} = 0, & \text{on } \partial \Omega \times (0,+\infty), \\
\hat{u}(x,0) = u_0 (x) - z(x), & \text{in } \Omega.
\end{array}\right.
\end{equation}

For $1\leq j\leq N-1$, we consider a large ball $B_R$ with 
$\Omega \subset B_R$. Inside this ball we take
$$
w (x,t) =  e^{R^2 \mu} e^{-\mu t} e^{-\mu \frac{r^2}{2}}.
$$
For large $\mu$, this function $w$ verifies
\begin{equation}\label{w.super.eq}
\left\lbrace\begin{array}{ll}
\displaystyle w_t - \lambda_1(D^2 w) = -\mu e^{R^2 \mu} e^{-\mu t} e^{-\mu \frac{r^2}{2}} + \mu e^{R^2 \mu} e^{-\mu t} e^{-\mu \frac{r^2}{2}}= 0, & \text{in } \Omega\times (0,+\infty), \\
w > 0, & \text{on } \partial \Omega \times (0,+\infty), \\
\displaystyle  w(x,0) =  e^{R^2 \mu} e^{-\mu \frac{r^2}{2}} \geq u_0 (x) - z(x), & \text{in } \Omega.
\end{array}\right.
\end{equation}
Hence, $w$ is a supersolution to \eqref{hat.u.eq} and then, by the comparison principle, we get
$$
\hat{u} (x,t) \leq w(x,t) = e^{R^2 \mu} e^{-\mu t} e^{-\mu \frac{r^2}{2}},
$$
for every $\mu$ large enough.
Then, for every $t> T = R^2/2$ we get
$$
\hat{u} (x,t) \leq \lim_{\mu \to \infty } e^{R^2 \mu} e^{-\mu t} e^{-\mu \frac{r^2}{2}} =0. 
$$

Hence, we have shown that when the boundary condition is the restriction of an affine function to the boundary, then there exists a finite time $T$ such that the solution to the evolution
problem lies below the stationary solution $z$, regardless the initial condition $u_0$, that is, it holds that
$$
u(x,t) \leq z(x)
$$
for every $x\in \Omega$ and every $t<T$ for $1\leq j \leq N-1$.

For $2\leq j \leq N$ the same argument proves that 
there exists a finite time $T$ such that the solution to the evolution
problem lies above the stationary solution.
Hence for $2\leq j\leq N-1$ there exists a finite time $T$ such that the solution to the evolution
problem coincides with the stationary solution.
This proves Theorem \ref{teo.comp.cur.intro}.
\end{proof}

Observe that for $j=1$, $u(x,t)=e^{-\mu_1t}\varphi_1(x)$ is a solution to the problem that do not become zero in finite time.
The same holds for $u(x,t)=e^{-\mu_1t}\psi_1(x)$ for $j=N$. Our next result shows that, in general, we can not expect
that all solutions lie below $z$ in the whole $\overline{\Omega}$ in finite time.

\begin{theorem}\label{teo.comportamiento.no.cur}
Let $\Omega$ be an open bounded domain in $\mathbb{R}^N$, and let $1\leq j\leq N$.
For any $x_0\in \Omega$, there exist $g$ and $u_0$ continuous in $\partial\Omega$ and $\overline{\Omega}$ respectively,
with $u_0|_{\partial\Omega} =g$, such that the solution of problem \eqref{convex envelope evolution} satisfies
$$
u (x_0,t) \geq z(x_0) + k e^{-\mu_1 t}, \qquad \text{for all} \ t>0, 
$$
where $\mu_1,k >0$ are two constants and $z$ is the solution of \eqref{convex envelope equation}.
\end{theorem}

We can obtain the analogous result for the inequality
$$
u (x_0,t) \leq z(x_0) - k e^{-\mu_1 t}.
$$

\begin{proof}
Consider, without loss of generality, that $x_0\in\Omega$ is the origin.
Take $r>0$ small enough such that the ball $B_r$ of radius $r$ and center at the origin satisfies $B_r\subset\subset \Omega$.

In the rest of the proof we will denote $\mathbb{R}^N = \mathbb{R}^j\times \mathbb{R}^{N-j}$, 
and we will write any point in $\mathbb{R}^N$ as $x= (x',x'') \in \mathbb{R}^j\times \mathbb{R}^{N-j}$.

Consider $B_r^j = B_r\cap \{x'' = 0\}$.
We observe that $B_r^j$ is a $j-$dimensional ball.
Therefore, as it is proven in \cite{BirindelliIshii},
there exists a positive eigenvalue $\mu_1$, with an eigenfunction $\psi_1$ which is continuous up to the boundary, such that
\begin{equation*}
\left\{ \begin{array}{ll}
-\lambda_j ( D^2\psi_1) = \mu_1\psi_1 & \text{in} \ B_r^j, \\
\psi_1 = 0 & \text{on} \ \partial B_r^j, \\
\psi_1 > 0  & \text{in} \ B_r^j.
\end{array}\right.
\end{equation*}

Consider $g$ a nonnegative continuous function defined on $\partial \Omega$ such that
\begin{equation}\label{g bigger than psi}
g(x',x'') \geq \psi_1(x'), \quad \text{for all} \ (x',x'')\in \partial\Omega,\ \text{with} \ x'\in B_r^j,
\end{equation}
and
\begin{equation}\label{g=0 in x''=0}
g(x',0) = 0 \quad \text{for all} \ (x',0) \in \partial\Omega \cap \{ x''=0\}.
\end{equation}
We note that this choice of $g$ is always possible since, if $x'\in B_r^j$, then $(x',0)\in B_r$, and since have considered 
$B_r\subset\subset \Omega$, we deduce $(x',0) \not\in \partial\Omega$.

For this choice of $g$, we claim that the solution of problem \eqref{convex envelope equation} satisfies
$$
z(x',x'') = 0, \qquad \text{in} \quad \Omega\cap \{x''=0\}.
$$
In order to prove this claim, we use the geometric interpretation of solutions to problem \eqref{convex envelope equation} given in \cite{BlancRossi}. 
Consider the $j-$dimensional subspace $\{x''=0\}$, and the $j-$dimensional domain $D:=\Omega\cap \{x''=0\}$. 
Following the ideas of \cite{BlancRossi}, 
the solution $z$ of \eqref{convex envelope equation} must satisfy
$$
z\leq z_D, \qquad \text{in} \quad D,
$$
where $z_D$ is the concave envelope of $g$ in $D=\Omega\cap \{x''=0\}$.
By the choice of $g$, using \eqref{g=0 in x''=0}, it follows that $z_D\equiv 0$.
The claim then follows from the maximum principle, since $g\geq 0$ in $\partial \Omega$.
In particular, we have
$$
z(0) = 0.
$$

Now, take $u_0$ a nonnegative continuous function in $\overline{\Omega}$ satisfying $u_0|_{\partial\Omega} = g$ and
\begin{equation}\label{u_0 bigger than psi_1}
u_0 (x',x'') \geq \psi_1 (x') \quad \text{for all} \ (x',x'')\in \Omega,\ \text{with} \ x'\in B_r^j.
\end{equation}

Consider the following function defined in the subdomain $\mathcal{Q} := \Big( \Omega \cap \{ x'\in B_r^j\} \Big) \times [0,+\infty)$:
$$
\underline{u} (x',x'',t) := \psi_1 (x') e^{-\mu_1 t}.
$$
We have
\begin{equation*}
\begin{array}{ll}
\underline{u}_t (x',x'',t) &= -\mu_1 \psi_1 (x') e^{-\mu_1 t}, \\
 \noalign{\vskip 1mm}
\lambda_j (D^2 \underline{u} (x',x'',t)) &= -\mu_1 \psi_1 (x') e^{-\mu_1 t}, \\
 \noalign{\vskip 1mm}
\underline{u} (x',x'',t) &\leq \psi_1 (x'),
\end{array}
\end{equation*}
in $\mathcal{Q}$. By \eqref{g bigger than psi} and \eqref{u_0 bigger than psi_1},
together with the comparison principle, we get
$$
\underline{u} (x',x'',t)\leq u(x',x'',t), \qquad \text{in} \ \mathcal{Q},
$$
and since $z(0) = 0$, we have 
$$
u(0,t) \geq z(0) + \psi_1(0) e^{-\mu_1 t}, \qquad \text{for all} \ t>0.
$$

\end{proof}

\subsection{Probabilistic arguments}

Here we will argue relating the value for our game and the value for the game \textit{a random walk for $\lambda_j$} introduced in \cite{BlancRossi}.
We call $z^\eps (x_0)$ the value of the game for the elliptic case (see \cite{BlancRossi}) considering the initial position $x_0$ and a length step of $\eps$.
This game is the same as the one described in Section \ref{sect-games}
 but now we do not take into account the time, that is, we do not stop when $t_k <0$ (and therefore we do not have that
the number of plays is a priori bounded by $\left\lceil 2T/\eps^2\right\rceil$).
We will call $x_\tau \not\in \Omega$ the final position of the token.
In what follows we will refer to the game described in Section \ref{sect-games} as the parabolic game while when we disregard time we
refer to the elliptic game. Notice that the elliptic DPP is given by 
\begin{equation*}
\left\{
\begin{array}{ll}
\displaystyle v^\eps (x) = \inf_{
{dim}(S)=j
} \sup_{
v\in S, |v|=1}
\left\{ \frac{1}{2} v^\eps \Big(x + \eps v\Big) + \frac{1}{2} v^\eps \Big(x - \eps v\Big)
\right\}  & x \in \Omega, \\
v^\eps (x) = g(x)  & x \not\in \Omega.
\end{array}\right.
\end{equation*}
Solutions to this DPP are stationary solutions (solutions independent of time) for the DPP
that correspond to the parabolic game. Let us recall it here,
\begin{equation*}
\left\{
\begin{array}{ll}
\displaystyle u^\eps (x,t) =
\inf_{{dim}(S)=j} \sup_{v\in S, |v|=1}
\left\{ \frac{1}{2} u^\eps \Big(x + \eps v,t-\frac{\eps^2}{2}\Big) + \frac{1}{2} u^\eps \Big(x - \eps v,t-\frac{\eps^2}{2}\Big)
\right\}  & (x,t) \in \Omega_T, \\[10pt]
u^\eps (x,t) = h(x,t)  & (x,t) \not\in \Omega_T.
\end{array}
\right.
\end{equation*}
Here we choose $h$ in such a way that it does not depend on $t$ (we can do this since we are assuming that $g$ does not depend on $t$). 

Our goal will be to show that there exist two positive constants $\mu$, depending only on $\Omega$, and $C$, depending on $u_0$, but both independent of 
$\eps$, such that
$$
\| u^\eps (\cdot,t) - v^\eps (\cdot)\|_\infty \leq C e^{-\mu t}.
$$

For the elliptic game, the strategies are denoted by $\tilde S_\I$ and $\tilde S_\II$. 
Given two strategies for the elliptic game, we can play the parabolic game according to those strategies by considering,
for all $t_0>0$,
\begin{equation}\label{equiv strategies}
\begin{array}{l}
S_\I{\left(t_0,x_0,x_1,\ldots,x_k\right)}=\tilde S_\I{\left(x_0,x_1,\ldots,x_k\right)} \\
 \noalign{\vskip 1mm}
S_\II{\left(t_0,x_0,x_1,\ldots,x_k,S\right)}=\tilde S_\II{\left(x_0,x_1,\ldots,x_k,S\right)}.
\end{array}
\end{equation}

When we attempt to do the analogous construction, building a strategy for the elliptic game given one for the parabolic game,
we require that the game sequences are not too long since the strategies for the parabolic game are only defined for $t_k>0$
(when $t_k\leq 0$ the parabolic game ends).
However, for any $t>0$, if we suppose that the game ends in less than $\left\lceil 2t/\eps^2\right\rceil$ steps,
i.e. $\tau < \left\lceil 2t/\eps^2\right\rceil$, then we have a bijection between strategies for the two games that have the same 
probability distribution for the game histories $(x_0, x_1,\dots, x_\tau)$.

The next lemma ensures that, in the parabolic game, the probability of the final payoff being given by the initial data goes to $0$ exponentially fast when $t \to +\infty$. In addition, we also prove that in the elliptic game, trajectories that take too long to exit the domain
have exponentially small probability.

\begin{lemma}\label{lemma games are not too long}
Let $\Omega$ be a bounded domain, $S_I,S_\II$ two strategies for the parabolic game
and $\tilde S_I,\tilde S_{II}$ two strategies for the elliptic game. We have, for any $t>0$,
$$
\mathbb{P}_{S_I,S_\II}^{x_0,t} [t_\tau \leq 0] \leq C e^{-\mu t}
\quad\text{ and }\quad
\mathbb{P}^{x_0}_{\tilde S_I,\tilde S_{II}}\left[\frac{\eps^2\tau}{2}\geq t\right]\leq C e^{-\mu t}
$$ 
where $\mu >0$ is a constant depending only on $\Omega$ and $C$ is another constant independent on the size of the steps, $\eps$. 
We recall that $\tau$ denotes the number of steps until the game ends.
\end{lemma}

\begin{proof} 
Take $B_R(x)$ such that $\Omega\subset B_R(x)$.
We start by proving the estimate for the elliptic game. 
Let $\tilde S_I, \tilde S_{II}$ be two strategies for this game.
As computed in \eqref{martingale}, $$M_k=|x_k-x|^2-\eps^2k$$ is a martingale.
By applying the optional stopping theorem, we obtain
\[
\eps^2\mathbb{E}^{x_0}_{\tilde S_I,\tilde S_{II}}[\tau]
=\mathbb{E}^{x_0}_{\tilde S_I,\tilde S_{II}}[|x_\tau-x|^2]
\leq R^2.
\]
Hence, we get
\[
\mathbb{E}^{x_0}_{\tilde S_I,\tilde S_{II}}\left[\frac{\eps^2\tau}{2}\right]
\leq \frac{R^2}{2}
\]
and we can show the bound
\[
\mathbb{P}^{x_0}_{\tilde S_I,\tilde S_{II}}\left[\frac{\eps^2\tau}{2}\geq t\right]
\leq \frac{R^2}{2t}.
\]
For $n\in\N$ by considering the martingale starting after $n$ steps, we can obtain
\[
\mathbb{P}^{x_0}_{\tilde S_I,\tilde S_{II}}\left[\frac{\eps^2\tau}{2}\geq \frac{\eps^2}{2}n+t\Big|\frac{\eps^2\tau}{2}\geq \frac{\eps^2}{2}n\right]
\leq \frac{R^2}{2t}.
\]
Hence, for $n,k\in\N$, applying this bound multiple times we obtain
\[
\begin{split}
\mathbb{P}^{x_0}_{\tilde S_I,\tilde S_{II}}\left[\frac{\eps^2\tau}{2}\geq \frac{\eps^2}{2}nk\right]
=&\ 
\mathbb{P}^{x_0}_{\tilde S_I,\tilde S_{II}}\left[\frac{\eps^2\tau}{2}\geq \frac{\eps^2}{2}nk\Big|\frac{\eps^2\tau}{2}\geq \frac{\eps^2}{2}n(k-1)\right]
\\
&\quad\times
\mathbb{P}^{x_0}_{\tilde S_I,\tilde S_{II}}\left[\frac{\eps^2\tau}{2}\geq \frac{\eps^2}{2}n(k-1)\Big|\frac{\eps^2\tau}{2}\geq \frac{\eps^2}{2}n(k-2)\right]
\\
&\quad\times\dots\times
\mathbb{P}^{x_0}_{\tilde S_I,\tilde S_{II}}\left[\frac{\eps^2\tau}{2}\geq \frac{\eps^2}{2}n\right]
\\
\leq& \left(\frac{R^2}{2(\frac{\eps^2n}{2})}\right)^k.
\end{split}
\]

For $\eps<\eps_0=1$ we consider $$\delta=\frac{R^2}{2e^{-1}}+\frac{1}{2}.$$ We have
\[
\mathbb{P}^{x_0}_{\tilde S_I,\tilde S_{II}}\left[\frac{\eps^2\tau}{2}\geq t\right]
\leq
\mathbb{P}^{x_0}_{\tilde S_I,\tilde S_{II}}\left[\frac{\eps^2\tau}{2}\geq  \frac{\eps^2}{2}\left\lfloor\frac{\delta 2}{\eps^2}\right\rfloor \left\lfloor\frac{t}{\delta}\right\rfloor
\right].
\]
By the above argument we obtain
\[
\mathbb{P}^{x_0}_{\tilde S_I,\tilde S_{II}}\left[\frac{\eps^2\tau}{2}\geq t\right]
\leq
\left(\frac{R^2}{2\left\lfloor\frac{\delta 2}{\eps^2}\right\rfloor\frac{\eps^2}{2}}\right)^{\left\lfloor\frac{t}{\delta}\right\rfloor}
\leq
\left(\frac{R^2}{2(\delta-\frac{\eps_0^2}{2})}\right)^{\frac{t}{\delta}-1}=e^{-\frac{t}{\delta}+1} .
\]
We have shown 
\[
\mathbb{P}^{x_0}_{\tilde S_I,\tilde S_{II}}\left[\frac{\eps^2\tau}{2}\geq t\right]\leq C e^{-\mu t}
\]
for $C=e$ and $\mu=\frac{1}{\delta}$.
The same bound holds for the parabolic game,  using the relation between the strategies given in \eqref{equiv strategies}.
That is,
\begin{eqnarray*}
\mathbb{P}_{S_I,S_\II}^{x_0,t} [t_\tau \leq 0]
&=&
\mathbb{P}^{x_0,t}_{S_I,S_{II}}\left[\frac{\eps^2\tau}{2}\geq t\right] = 1 - \mathbb{P}^{x_0,t}_{S_I,S_{II}} \left[ \tau < \frac{2t}{\eps^2} \right] \\
&=& 1 - \mathbb{P}^{x_0}_{\tilde S_I,\tilde S_{II}} \left[ \tau < \frac{2t}{\eps^2} \right]
\leq  C e^{-\mu t}.
\end{eqnarray*}
The use of the equivalence \eqref{equiv strategies} between strategies of the two games is justified because we are computing
the probability of the number of steps being less than $\left\lceil 2t/\eps^2\right\rceil$.   
\end{proof}

Using Lemma \ref{lemma games are not too long}, we are able to prove
that, as happens for the evolution PDE (see the previous subsection), 
also in the game formulation, the asymptotic behaviour of the value function
as $t$ goes to infinity is given by the value of the elliptic game (that is, by the stationary solution of the game).
Notice that in the probabilistic approach we obtain a bound for 
$\|u(\cdot,t) - z(\cdot)\|_\infty $ of the form $C\|u_0\|_\infty e^{-\mu t}$.
However, we do not have that $\mu$ comes from an eigenvalue problem
but from the exponential bounds obtained in Lemma \ref{lemma games are not too long}.

\begin{proposition}
There exists $\mu >0$, a constant depending only on $\Omega$, and $C>0$ depending on $u_0$, such that
$$
\| u^\eps (\cdot,t) - v^\eps (\cdot)\|_\infty \leq C e^{-\mu t},
$$
where $u^\eps$ and $v^\eps$ are the value functions for the parabolic and the elliptic game, respectively.

Moreover, as a consequence of this exponential decay, we obtain that the solution $u$ of the problem \eqref{convex envelope evolution} and the convex envelope $z(x)$ of $g$ in $\Omega$ satisfy
$$
\|u(\cdot,t) - z(\cdot)\|_\infty \leq C e^{-\mu t}. 
$$ 
\end{proposition}

\begin{proof}
Recall the payoff function $h$ defined in \eqref{h def}, here do not depend on $t$.
For any $(x_0,t_0)\in \Omega\times (0,+\infty)$ fixed, we have
\begin{equation}\label{u eps 1}
\begin{array}{rcl}
u^\eps (x_0,t_0) &=& \displaystyle\inf_{S_I} \sup_{S_\II}  \mathbb{E}_{S_I,S_\II}^{x_0,t_0} \big[ h(x_\tau,t_\tau) \big]  \\
 \noalign{\vskip 1mm}
 &=& \displaystyle
 \inf_{S_I} \sup_{S_\II} \left\{ \mathbb{E}_{S_I,S_\II}^{x_0,t_0} \big[ g(x_\tau)| t_\tau > 0 \big]
 \mathbb{P}_{S_I,S_\II}^{x_0,t_0} (t_\tau > 0) \right.\\
 \noalign{\vskip 1mm}
\displaystyle
 & & + \left. \mathbb{E}_{S_I,S_\II}^{x_0,t_0} \big[ u_0(x_\tau)| t_\tau \leq 0 \big] \mathbb{P}_{S_I,S_\II}^{x_0,t_0} (t_\tau \leq 0) \right\rbrace \\
 \noalign{\vskip 1mm}
 &\leq& \displaystyle\inf_{S_I} \sup_{S_\II} \mathbb{E}_{S_I,S_\II}^{x_0,t_0} \big[ g(x_\tau)| t_\tau > 0 \big]  +(\|g\|_\infty+\|u_0\|_\infty) \sup_{S_\I,S_\II} \mathbb{P}_{S_I,S_\II}^{x_0,t_0} (t_\tau \leq 0)
\end{array}
\end{equation}
and
\begin{equation}\label{u eps 2}
\begin{array}{rcl}
u^\eps (x_0,t_0)  
&\geq& \displaystyle\inf_{S_I} \sup_{S_\II} \mathbb{E}_{S_I,S_\II}^{x_0,t_0} \big[ g(x_\tau)| t_\tau > 0 \big]
-(\|g\|_\infty+\|u_0\|_\infty) \sup_{S_\I,S_\II} \mathbb{P}_{S_I,S_\II}^{x_0,t_0} (t_\tau \leq 0).
\end{array}
\end{equation}

Now, let $z^\eps(x_0)$ be the value of the elliptic game considering as payoff function  the same function $g$ as before. We have
\begin{equation}\label{z eps 1}
\begin{array}{rcl}
z^\eps (x_0) &=& \displaystyle
\inf_{\tilde S_I} \sup_{\tilde S_\II} \left\{ \mathbb{E}_{\tilde{S}_I,\tilde S_\II}^{x_0} \big[g(x_\tau)| \tau < 2t_0/\eps^2\big]
\mathbb{P}_{\tilde{S}_I,\tilde S_\II}^{x_0} (\tau < 2t_0/\eps^2) \right. \\
 \noalign{\vskip 1mm}
& & + \left. \mathbb{E}_{\tilde{S}_I,\tilde S_\II}^{x_0} \big[g(x_\tau)| \tau \geq 2 t_0/\eps^2\big] \mathbb{P}_{\tilde{S}_I,\tilde S_\II}^{x_0} (\tau \geq 2 t_0/\eps^2) \right\rbrace \\
 \noalign{\vskip 1mm}
&\leq & \displaystyle\inf_{\tilde S_I} \sup_{\tilde S_\II}  \mathbb{E}_{\tilde{S}_I,\tilde S_\II}^{x_0} \big[g(x_\tau)| \tau < 2t_0/\eps^2\big]
+ \|g\|_\infty \displaystyle\sup_{\tilde{S}_I,\tilde S_\II} \mathbb{P}_{\tilde{S}_I,\tilde S_\II}^{x_0} (\tau \geq t_0/\eps^2).
\end{array} 
\end{equation}
and
\begin{equation}\label{z eps 2}
\begin{array}{rcl}
z^\eps (x_0) 
&\geq & \displaystyle\inf_{\tilde S_I} \sup_{\tilde S_\II}  \mathbb{E}_{\tilde{S}_I,\tilde S_\II}^{x_0} \big[g(x_\tau)| \tau < 2t_0/\eps^2\big]
- \|g\|_\infty \displaystyle\sup_{\tilde{S}_I,\tilde S_\II} \mathbb{P}_{\tilde{S}_I,\tilde S_\II}^{x_0} (\tau \geq t_0/\eps^2).
\end{array} 
\end{equation}

Given $t_0>0$ in the parabolic game, if we suppose that $\tau< 2t_0/\eps^2$ in both games, we have an equivalence between the strategies of both games, regardless what happens after step $\left\lfloor 2t_0/\eps^2\right\rfloor$.
That is,
\[
\inf_{\tilde S_I} \sup_{\tilde S_\II}  \mathbb{E}_{\tilde{S}_I,\tilde S_\II}^{x_0} \big[g(x_\tau)| \tau < 2t_0/\eps^2\big]
=
\inf_{S_I} \sup_{S_\II} \mathbb{E}_{S_I,S_\II}^{x_0,t_0} \big[ g(x_\tau)| t_\tau > 0 \big].
\]

Now, combining \eqref{u eps 1}, \eqref{u eps 2}, \eqref{z eps 1} and \eqref{z eps 2}, we obtain
\begin{equation}\label{u-z estimate 2}
\left| u^\eps (x_0,t_0) - z^\eps (x_0)\right| \leq 
2\|u_0\|_\infty \displaystyle
\left(
\sup_{\tilde{S}_I,\tilde S_\II} \mathbb{P}_{\tilde{S}_I,\tilde S_\II}^{x_0} (\tau \geq 2t_0/\eps^2)
+
\sup_{S_\I,S_\II}  \mathbb{P}_{S_I,S_\II}^{x_0,t_0} (t_\tau \leq 0)
\right).
\end{equation}

Applying Lemma \ref{lemma games are not too long}, for $\eps<\eps_0=1$, we have
$$
|u^\eps(x_0,t_0) - z^\eps(x_0)| \leq 4\|u_0\|_\infty C e^{-\mu t_0},
$$
for some $\mu$ depending only on $\Omega$.
Letting $\eps\to 0$ and using the uniform convergence of $u^\eps(x_0,t_0)$ and $z^\eps(x_0)$ to $u(x_0,t_0)$ and $z(x_0)$,
respectively, we obtain
$$
|u(x_0,t_0) - z(x_0)| \leq 4\|u_0\|_\infty C e^{-\mu t_0}.
$$
This completes the proof.
\end{proof}

Now, assume that there is an affine function $\pi$ such that $g=\pi$ for $x\not\in \Omega$.
In this case, we have that $\pi(x_k)$ is a martingale.
Hence, under a strategy that forces the game to end outside $\Omega$, we obtain that $\mathbb{E}_{S_I,S_\II}^{x_0,t_0}[h(x_\tau,t_\tau)]=\pi(x_0)$.

Suppose $1\leq j \leq N-1$, $\Omega\subset B_R(x)$ and $g\equiv \pi$.
Player $I$ can choose $S$ at every step in such a way that it is normal to $x-x_k$, hence $v\in S$ is normal to $x-x_k$, we have
\[
|x-x_{k+1}|^2= |x-x_k-v\eps|^2=|v\eps|^2+|x-x_k|^2=\eps^2+|x-x_k|^2.
\]
If Player $I$ plays with this strategy, we obtain $|x-x_k|^2=k\eps^2+|x-x_0|^2$.
Since $\Omega\subset B_R(x)$, $|x-x_k|^2\leq R^2$ for every $x_k\in\Omega$, and hence the game ends after at most
\[
\frac{R^2-|x-x_0|^2}{\eps^2}
\]
turns.
Hence, it holds that
$$
u(x,t) \leq \pi(x)
$$
for every $x\in \Omega$ and every $t>T=2R^2$.

Analogously, if $2\leq j \leq N$, Player II can choose $v\in S$ such that $v$ is normal to $x-x_k$ (because the intersection of $S$ and the $N-1$ dimensional normal space to $x-x_k$ is not empty).
By the same arguments used before, we can show that 
$$
u(x,t) \geq \pi(x)
$$
for every $x\in \Omega$ and every $t>T=2R^2$.

Hence, we have shown that, for $2\leq j \leq N-1$
$$u(x,t)=\pi(x)$$
for every $x\in \Omega$ and every $t>T=2R^2$.
Note that this argument can be considered as a proof of Theorem \ref{teo.comp.cur.intro} based on the game strategies.

We can obtain a similar result when $g=\pi$ in a half-space.
Suppose that $h=\pi$ for every $x\in \{x\in\Omega^c: x\cdot w>\theta\}$ for a given $w\in\R^N$ of norm 1 and $\theta\in\R$. 
Given $y\in \{x\in\Omega: x\cdot w>\theta\}$ we can choose $\xi\in\R^N$ and $r>0$ such that $\{x\in\Omega: x\cdot w\leq \theta\}\subset B_r(\xi)$ and $y\not\in B_r(\xi)$ as depicted in Figure~\ref{bigball}.

\begin{figure}
\begin{tikzpicture}
    \draw [rotate around={45:(0,0)},red] (0,0) ellipse (2cm and 1cm);
    \draw (-0.5,-2) -- (-0.5,2.5);
    \draw [->] (-0.5,2) -- (0.5,2);
    \node at (0,2.2){$w$};
    \node at (2,0){$h=\pi$};
    \node at (0.2,0){$y$};
    \node at (0,0){$\cdot$};
    \draw [blue] (-1,-2.35) arc (-45:45:3cm);
\end{tikzpicture}
\caption{Here $\partial\Omega$ is in red and $\partial B_r(\xi)$ in blue.}
\label{bigball}
\end{figure}
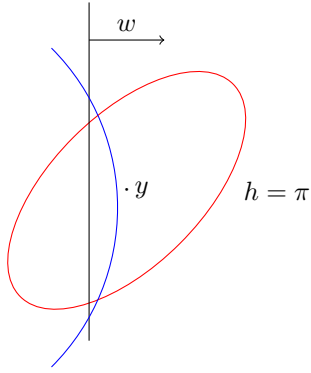

Now, arguing in the same way as before, we can consider the strategies that give a vector $v$ normal to $x_k-\xi$.
Hence, in the case $1\leq j \leq N-1$ we can prove that
$$
u(x,t) \leq \pi(x)
$$
for every $y\in \{x\in\Omega: x\cdot w>\theta\}$ and every $t$ large enough (for instance we can take $t>2r^2$ where $r$ is the radius of the ball described before, that depends on $x$).
Note that the closer is $y$ to the hyperplane $x\cdot w = \theta$, the longer we will have to wait for having the above inequality. 

In the case $2\leq j \leq N$, with analogous arguments, we can also show that we have the reverse inequality, that is,
$$
u(x,t) \geq \pi(x)
$$
for every $x\in \{x\in\Omega: x\cdot w>\theta\}$ and every $t$ large enough.

Next, we present an example that illustrates the result of Theorem \ref{teo.comportamiento.no.cur}.
Although it is possible to give a more general argument, giving rise to an alternative proof of this theorem
based only on probabilistic arguments,
we restrict ourselves to this example to clarify the exposition.

\begin{ejem} \label{example.final}
Consider the parabolic game for $\lambda_j$ in a ball $B_R$ centered at the origin, and take as initial and boundary data two functions $u_0 (x',x'')$ and $g(x',x'')$,
with $(x',x'')\in \mathbb{R}^{j}\times\mathbb{R}^{N-j}$, such that
$$ u_0 > 0, \quad \text{in} \ \Omega, \quad \text{and} \quad g(x',x'') = |x''|, \quad \text{for all} \ (x',x'')\in \R^N\setminus \Omega.$$   
For this choice of $g$, we claim that the solution of problem \eqref{convex envelope equation} satisfies
$$
z(x',x'') = 0, \qquad \text{in} \quad \Omega\cap \{x''=0\}.
$$
In order to prove this claim, we use the geometric interpretation of solutions to problem \eqref{convex envelope equation} given in \cite{BlancRossi}. 
Consider the $j-$dimensional subspace $\{x''=0\}$, and the $j-$dimensional domain $D:=\Omega\cap \{x''=0\}$. 
Following the ideas of \cite{BlancRossi}, 
the solution $z$ of \eqref{convex envelope equation} must satisfy
$$
z\leq z_D, \qquad \text{in} \quad D,
$$
where $z_D$ is the concave envelope of $g$ in $D=\Omega\cap \{x''=0\}$.
By the choice of $g$, it follows that $z_D\equiv 0$.
The claim then follows from the maximum principle, since $g\geq 0$ in $\partial \Omega$.

Now, let us prove that for any $x_0\in \Omega\cap \{x'' = 0\}$ and $t_0>0$, we have
$$
u^\eps (x_0,t_0) = \displaystyle\inf_{S_I} \sup_{S_\II}  \mathbb{E}_{S_I,S_\II}^{x_0,t_0} \big[ h(x_\tau,t_\tau) \big] > 0.
$$

Let $x_0\in \Omega\cap \{x'' = 0\}$. 
Since $u_0\geq 0$, if $u^\eps (x_0,t_0)= 0$, Player I should have a strategy such that whatever Player II does, the final payoff is $0$ with probability 1.
Since $u_0$ vanishes only on $\partial \Omega\cap \{x'' = 0\}$, Player I needs to make sure that $x_k$ reaches this set before the game comes to end.

We claim that the only strategy Player I can follow is to choose the $j-$dimensional subspace $\{x'' = 0\}$ at every step.
Indeed, if at some step, $x_k$ leaves this subspace, the probability of never coming back, and then the final payoff being non-zero, is positive.

Once Player I has fixed this only possible strategy to obtain zero as final payoff, Player II can choose any unitary vector
in the subspace $\{x'' = 0\}$, and plays always with the same vector. Playing with these strategies, the game is reduced to a random walk in a segment,
and it is well known that for this process, the probability of not reaching the extremes of the segment in less than
$\left\lceil 2t_0/\eps^2\right\rceil$ steps is strictly positive for any $t_0>0$ (in fact, it is uniformly bounded below). 
Since the initial condition verifies $u_0>0$ in $\Omega$, we conclude that 
the value of the game is also strictly positive at $(x_0,t_0)$, moreover, it is bounded below, $u^\eps (x_0,t_0)>c>0$, for any $x_0\in \Omega\cap \{x'' = 0\}$ and $t_0>0$ independently of $\eps$.
Then, 
$u^\eps (x_0,t_0)$, and hence its limit as $\eps \to 0$, $u(x_0,t_0)$, does not lie below the stationary solution $z$ in finite time. 

Finally, notice that from our previous arguments, we have that for any point $x_0\in \Omega\setminus \{x'' = 0\}$ there is a finite time
$t_0$ (that depends on $x_0$) such that $u^\eps (x_0,t) = z(x_0)$ for every $t\geq t_0$.
\end{ejem}

\medskip

{\bf Acknowledgements.} Partially supported by CONICET grant PIP GI No 11220150100036CO
(Argentina), by  UBACyT grant 20020160100155BA (Argentina) and by MINECO MTM2015-70227-P
(Spain). 
CE partially supported by Sorbonne Universit\'e, Laboratoire Jaques-Louis Lions (LJLL)
Paris, France.


\end{document}